\documentclass{aic}

\usepackage{amsmath,amsfonts,amssymb,amsthm,fancyhdr,bm,mathtools,enumitem}
\usepackage[capitalize]{cleveref}
\usepackage{tikz}
\usepackage[english]{babel}

\tikzset{vert/.style={draw, fill=black, circle, inner sep=2pt}}

\newtheorem{thm}{Theorem}[section]
\newtheorem{lem}[thm]{Lemma}

\newtheorem{conj}[thm]{Conjecture}
\newtheorem{claim}[thm]{Claim}
\theoremstyle{definition}
\newtheorem{Def}[thm]{Definition}
\newtheorem*{rem}{Remark}

\crefname{equation}{equation}{equations}
\crefname{lem}{Lemma}{Lemmas}
\crefname{claim}{Claim}{Claims}
\crefname{thm}{Theorem}{Theorems}
\crefname{conj}{Conjecture}{Conjectures}

\newlist{lemenum}{enumerate}{1}
\setlist[lemenum]{label=(\alph*), ref=\thelem(\alph*)}
\crefalias{lemenumi}{lemma}

\newcommand\wh[1]{\widehat{#1}}
\newcommand\wt[1]{\widetilde{#1}}
\newcommand\ab[1]{\lvert#1\rvert}
\newcommand{\flo}[1]{\lfloor #1 \rfloor}
\newcommand{\cei}[1]{\lceil #1 \rceil}

\newcommand\Q{\mathcal{Q}}
\newcommand\K{\mathcal{K}}
\newcommand\h{\mathcal{H}}

\numberwithin{equation}{section}

\aicAUTHORdetails{%
  title = {Ramsey multiplicity and the Tur\'an coloring}, 
  author = {Jacob Fox and Yuval Wigderson},
  plaintextauthor = {Jacob Fox, Yuval Wigderson},
  plaintexttitle = {Ramsey multiplicity and the Turan coloring}, 
}   

\aicEDITORdetails{%
   year={2023},
   number={2},
   received={20 July 2022},   
   published={15 July 2023},  
   doi={10.19086/aic.2023.2},      
}   

\begin{document}

\begin{frontmatter}[classification=text]


\author[jacob]{Jacob Fox\thanks{Research supported by a Packard Fellowship and by NSF Awards DMS-1953990 and DMS-2154129.}}
\author[yuval]{Yuval Wigderson\thanks{Research supported by NSF GRFP Grant DGE-1656518, ERC Consolidator Grant 863438, ERC Starting Grant 101044123, and NSF-BSF Grant 20196.}}

\begin{abstract}
	Extending an earlier conjecture of Erd\H os, Burr and Rosta conjectured that among all two-colorings of the edges of a complete graph, the uniformly random coloring asymptotically minimizes the number of monochromatic copies of any fixed graph $H$.  This conjecture was disproved independently by Sidorenko and Thomason. The first author later found quantitatively stronger counterexamples, using the \emph{Tur\'an coloring}, in which one of the two colors spans a balanced complete multipartite graph.

	We prove that the Tur\'an coloring is extremal for an infinite family of graphs, and that it is the unique extremal coloring. 
	This yields the first determination of the Ramsey multiplicity constant of a graph for which the Burr--Rosta conjecture fails.

	We also prove an analogous three-color result. In this case, our result is conditional on a certain natural conjecture on the behavior of two-color Ramsey numbers.
\end{abstract}
\end{frontmatter}

\section{Introduction}
\subsection{Background}
Let $H$ be a fixed graph on $t$ vertices. If $\chi:E(K_n) \to \{\text{red},\text{blue}\}$ is a two-coloring of the edges of a complete graph $K_n$, then we denote by $m(H,\chi)$ the number of monochromatic \emph{labeled} copies of $H$ in the coloring $\chi$. We also denote by $m(H,n)$ the minimum of $m(H,\chi)$ over all two-colorings of $E(K_n)$. A simple averaging argument shows that the sequence $c(H,n)\coloneqq m(H,n)/(n)_t$ is non-decreasing, where $(n)_t=n(n-1)\dotsb(n-t+1)$ is the falling factorial. Since this sequence is contained in $[0,1]$, we see that the limit
\[
	c(H) \coloneqq \lim_{n \to \infty} c(H,n)
\]
is well defined. This constant, $c(H)$, is called the \emph{Ramsey multiplicity constant} of $H$. It is often helpful to think of $c(H)$ probabilistically: it equals the asymptotic minimum over all two-colorings of $E(K_n)$ of the probability that a random map $V(H) \to V(K_n)$ yields a monochromatic copy of $H$. We note that by Ramsey's theorem, $c(H,n)$ is strictly positive for any graph $H$ and any sufficiently large $n$, which implies that $c(H)>0$ since the sequence is non-decreasing.

The earliest result about Ramsey multiplicity, preceding the definition, is due to Goodman \cite{MR107610} from 1959, and implies that $c(K_3) = \frac 14$. Since a random two-coloring of $E(K_n)$ has $\frac 14 (n)_3$ monochromatic labeled copies of $K_3$ in expectation, the random coloring shows the existence of a coloring matching Goodman's lower bound. A few years later, Erd\H os \cite{MR151956} conjectured that the random coloring asymptotically minimizes the number of monochromatic copies of $K_t$ for all $t$, and thus that $c(K_t) = 2^{1- \binom t2}$. Burr and Rosta \cite{MR595601} formulated a natural generalization of Erd\H os's conjecture, namely that $c(H) = 2^{1-e(H)}$ for every graph $H$. Since a random coloring of $K_n$ contains $2^{1-e(H)}(n)_t$ monochromatic labeled copies of $H$ in expectation,  we have $c(H) \leq 2^{1-e(H)}$, and the Burr--Rosta conjecture simply says that this upper bound is tight, i.e.\ that a random coloring asymptotically minimizes the number of monochromatic copies of $H$.

The Burr--Rosta conjecture was disproved by Sidorenko \cite{MR1033422}, who showed that it is false when $H$ is a triangle with a pendant edge. At roughly the same time, even the weaker conjecture of Erd\H os was disproved by Thomason \cite{MR991659}, who constructed, for every $t \geq 4$, colorings of $K_n$ with asymptotically fewer than $2^{1- \binom t2}(n)_t$ monochromatic labeled copies of $K_t$. A more dramatic counterexample to the Burr--Rosta conjecture was found by the first author \cite{MR2374234}, who constructed an infinite family of graphs $H$ with $m$ edges and with $c(H) \leq 2^{-\Omega(m \log m)}$, thus showing that the Burr--Rosta conjecture can be very far from true. 

The key to the examples from \cite{MR2374234} comes from the \emph{Tur\'an coloring} of $K_n$, namely the
coloring where the blue graph is a balanced complete $(k-1)$-partite graph and where the red graph consists of $k-1$ disjoint cliques, each of order $\lfloor\frac n{k-1}\rfloor$ or $\lceil\frac n{k-1}\rceil$. Let $r$ be the remainder when $n$ is divided by $k-1$, so there are $k-1-r$ cliques of order $\flo{\frac{n}{k-1}}$, and further $r$ cliques of order $\cei{\frac{n}{k-1}}$. Since the blue graph has chromatic number $k-1$, there can be no blue copies of any graph $H$ with chromatic number at least $k$. Moreover, if $H$ is connected, then any red copy of $H$ must appear inside one of the parts. If we let $t$ be the number of vertices in $H$, then the number of labeled copies of $H$ in the Tur\'an coloring is exactly
\[
	(k-1-r) \left(\left \lfloor{\frac{n}{k-1}}\right \rfloor\right)_t + r \left(\left \lceil{\frac{n}{k-1}}\right \rceil\right)_t \leq (k-1)^{1-t} (n)_t.
\]
Said differently, one sees that the probability that a random injection $V(H) \to V(K_n)$ has its image contained entirely within one of the $k-1$ parts is at most $(k-1)^{1-t}$. Thus, the Tur\'an coloring shows that for any connected graph $H$ with $t$ vertices and chromatic number $k$, we have
\begin{equation}\label{eq:turan-coloring-ub-mhn}
	m(H,n) \leq 
	 (k-1)^{1-t} (n)_t
\end{equation}
and therefore
\begin{equation}\label{eq:turan-coloring-ub-ch}
	c(H) \leq (k-1)^{1-t} =2^{-\Omega(t \log k)}.
\end{equation}
One way of picking a graph $H$ which optimizes this construction is to start with a clique of order $k$, and then to add $t-k$ pendant edges to this clique. Thus, $H$ has $\binom k2 + (t-k)$ edges,
and (\ref{eq:turan-coloring-ub-ch}) shows that $c(H)$ is super-exponentially small in the number of edges of $H$ if $k=\omega(1)$ and $t = \omega(k^2/{\log k})$.

Graphs $H$ which satisfy the Burr--Rosta conjecture, i.e.\ those with $c(H) = 2^{1-e(H)}$, are called \emph{common}. Sidorenko's conjecture (see e.g.\ \cite{MR1225933}), a major open problem in extremal graph theory, states that for any bipartite graph $H$ and any $p \in [0,1]$, a random graph with edge density $p$ has in expectation asymptotically the minimum number of copies of $H$ over all graphs of the same order and edge density.
It is straightforward to show that $H$ is common if $H$ satisfies Sidorenko's conjecture, so Sidorenko's conjecture implies that all bipartite graphs are common. Sidorenko's conjecture has been verified in many cases (e.g.\ \cite{MR4237083,MR2607540,MR3456171,MR3893193,MR2738996,1406.6738,1107.1153}), yielding a large class of bipartite graphs which are known to be common. Additionally, some non-bipartite graphs are known to be common, such as odd cycles \cite{MR1033422}, the five-wheel \cite{MR2959863}, the seven-wheel \cite{GrLeLiVo}, and certain graphs of arbitrary chromatic number \cite{2206.05800}. On the other hand, it is also known that most graphs are uncommon, since Jagger, \v S\v tov\'i\v cek, and Thomason \cite{MR1394515} proved that any graph containing $K_4$ is uncommon.

In general, we do not understand the Ramsey multiplicity constants of uncommon graphs. For instance, Thomason \cite{MR991659} proved that $K_4$ is uncommon by exhibiting an explicit coloring witnessing $c(K_4) < \frac 1{33}$. Using flag algebras, Nie\ss{} and Sperfeld \cite{1207.4714,1106.1030} independently showed that $c(K_4)> \frac 1{35}$. There have been very recent improvements \cite{GrLeLiVo,2206.04036} to the upper and lower bounds on $c(K_4)$, but it remains unclear what the true value of $c(K_4)$ is, as well as what the asymptotically optimal colorings look like.

\subsection{Our results}

The discussion above is actually somewhat prototypical of Ramsey theory. Namely, both the random coloring and the Tur\'an coloring arise naturally in many Ramsey-theoretic contexts, and sometimes one of the two is (asymptotically) tight. For instance, the study of Ramsey goodness (see e.g.\ \cite{MR2520282}) asks when the Tur\'an coloring yields optimal bounds for certain off-diagonal Ramsey numbers, and there are diagonal Ramsey problems (see e.g.\ \cite{MR4115773,CoFoWi}) for which the random coloring is known to be asymptotically optimal. In instances where neither the random coloring nor the Tur\'an coloring is tight, our knowledge is frequently limited, and there are many problems for which no optimal structure is known or conjectured.

Since common graphs have been extensively studied, and since the field of Ramsey goodness is very rich, it is perhaps surprising that there are no results about when the Tur\'an coloring is optimal for the Ramsey multiplicity problem. In this paper, we study this problem. We make the following definition.
\begin{Def}\label{def:bonbon}
	A connected graph $H$ on $t$ vertices is called a \emph{bonbon} if, for every sufficiently large $n$, the Tur\'an coloring of $E(K_n)$ with $\chi(H)-1$ parts is the unique two-edge-coloring of $K_n$ with the minimum number of monochromatic copies of $H$, up to permuting the colors and the vertices. 
\end{Def}
In particular, if $H$ is a bonbon, then $c(H) = (\chi(H)-1)^{1-t}$.
The word \emph{bonbon} is meant to be reminiscent of the ``good'' terminology in the study of Ramsey goodness, as well as in the recent work \cite{2012.12646,2006.03756} on \emph{Tur\'an goodness}, which studies when the Tur\'an graph is the unique extremizer for a certain extremal problem.
Our first main result proves that all graphs in a certain infinite family are bonbons.
To define this infinite family, we recall two standard definitions. First, a graph is called \emph{$k$-critical} if it has chromatic number $k$, but there is an edge\footnote{Note that some authors use the term \emph{$k$-critical} to mean something different, namely that the deletion of \emph{any} edge lowers the chromatic number.} whose deletion lowers the chromatic number. Second, given a graph $H_0$, we say that another graph $H$ is \emph{obtained from $H_0$ by adding pendant edges} if $H_0$ is an induced subgraph of $H$, all vertices in $V(H) \setminus V(H_0)$ have degree $1$, and their unique neighbor lies in\footnote{We stress that we make no assumption about which vertices of $H_0$ are the neighbors of the new vertices.} $V(H_0)$. 

The graphs that we prove are bonbons are 
obtained from a $k$-critical graph by adding sufficiently many pendant edges.
A result of Simonovits \cite{MR337690} states that if $H$ is $k$-critical and if $n$ is sufficiently large, then the Tur\'an graph is the unique $n$-vertex $H$-free graph with the maximum number of edges; because of this, $k$-criticality arises naturally in many questions involving the Tur\'an graph or Tur\'an coloring.

\begin{thm}\label{thm:main}
	Let $k \geq 4$, and let $H_0$ be a $k$-critical graph with $h$ vertices. Let $H$ be obtained from $H_0$ by adding $t-h$ pendant edges, where $t \geq (1000kh)^{10} h^{10k}$. Then $H$ is a bonbon. In particular, $c(H)=(k-1)^{1-t}$.
\end{thm}
This gives the first examples of uncommon graphs whose Ramsey multiplicity constants are known exactly. 
Natural examples of graphs $H$ to which \cref{thm:main} applies are \emph{starbursts} (obtained from $K_k$ by adding an equal number of pendant edges to each vertex of $K_k$; see \cite{2111.05420} for the terminology) and \emph{pineapples} (obtained from $K_k$ by adding pendant edges to a single vertex; see e.g.\ \cite{MR4016583} for the terminology). In both of these examples, we set $H_0=K_k$, which is a $k$-critical graph.

All three of the assumptions of \cref{thm:main}, namely that $H_0$ is $k$-critical, $k$ is at least $4$, and $t$ is sufficiently large, are necessary, as we now explain. First, suppose that $H_0$ is not $k$-critical. Then if we begin with the Tur\'an coloring and recolor a single edge inside one part from red to blue, then we strictly decrease the number of red copies of $H$. On the other hand, this new coloring still contains no blue $H_0$, and thus no blue $H$, for an embedding of $H_0$ into the blue graph precisely corresponds to a proper $(k-1)$-coloring of a graph obtained from $H_0$ by deleting an edge. This shows that if $H_0$ is not $k$-critical, then $H$ is not a bonbon. 

Next, if $k=3$, the Tur\'an coloring yields an upper bound $c(H) \leq 2^{1-t}$ as $H$ has $t$ vertices. Since $H$ is a connected graph with at least one cycle, we have that $t \geq e(H)$, and thus this upper bound is no stronger than the upper bound $c(H) \leq 2^{1-e(H)}$ given by the random coloring. If $H_0$ has at least two cycles, then this shows that $H$ cannot be a bonbon, as the random coloring yields a stronger upper bound than the Tur\'an coloring. Additionally, if $H_0$ has exactly one cycle, then if $H$ is a bonbon, $H$ is in particular common. However, Jagger, \v{S}\v{t}ov\'{\i}\v{c}ek, and Thomason \cite[Theorem 4]{MR1394515} showed that if $H$ is obtained from a non-bipartite graph by adding sufficiently many pendant edges, then $H$ is uncommon. 

Finally, the assumption in \cref{thm:main} that we add sufficiently many pendant edges is also necessary, for a similar reason. For example, if $H_0=K_k$, and if we add $o(k^2/\log k)$ pendant edges, then the random coloring will again have fewer monochromatic copies of $H$ than the Tur\'an coloring.

\subsection{More colors}
Our second main result concerns Ramsey multiplicity for more colors, which we now introduce. For any integer $q \geq 3$, let $m_q(H,n)$ denote the minimum number of monochromatic labeled copies of $H$ in any $q$-coloring of the edges of $K_n$. As in the case of two colors, a simple averaging argument shows that the limit
\[
	c_q(H) = \lim_{n \to \infty} \frac{m_q(H,n)}{(n)_t}
\]
exists, and is called the \emph{$q$-color Ramsey multiplicity constant} of $H$. As before, $c_q(H)$ is strictly positive by Ramsey's theorem and by the non-decreasing property. For $q \geq 3$, our understanding of $q$-color Ramsey multiplicity is even more limited than it is for two colors. As in the case of two colors, a random coloring shows that $c_q(H) \leq q^{1-e(H)}$, and graphs $H$ for which this bound is tight are called \emph{$q$-common}. Sidorenko's conjecture again implies that any bipartite graph is $q$-common for all $q$, and a simple product coloring shows that any non-bipartite $H$ is \emph{not} $q$-common for all sufficiently large $q$. In the other direction, Kr\'a\v l, Noel, Norin, Volec, and Wei \cite{2006.09422} recently showed that for every $q$, there exist non-bipartite $q$-common graphs. They also showed that if a graph $H$ is $q$-common for all $q$, then $H$ must satisfy Sidorenko's conjecture.

Just as the Tur\'an coloring yields a natural upper bound on the two-color Ramsey multiplicity constant for a graph, there is an analogous construction that upper-bounds $c_q(H)$ for $q \geq 3$. Let $r_{q-1}(k)$ denote the $(q-1)$-color Ramsey number of $K_k$, namely the least $r$ so that any $(q-1)$-coloring of $E(K_r)$ contains a monochromatic copy of $K_k$. Then we define the \emph{$q$-color Ramsey-blowup coloring} with parameter $k$ to be the $q$-coloring of $E(K_n)$ where we partition $V(K_n)$ into $r_{q-1}(k)-1$ equally-sized parts, color the edges between the parts by blowing up a $(q-1)$-coloring of $E(K_{r_{q-1}(k)-1})$ without a monochromatic $K_k$, and color all edges within the parts with the $q$th, unused, color. If $H$ is a graph with clique number $k$, then there can be no monochromatic copy of $H$ in the first $q-1$ colors\footnote{Note that in contrast to the two-color case, here we require $H$ to have clique number $k$, rather than chromatic number $k$. The reason is that in a general coloring of $K_{r_{q-1}(k)-1}$ with no monochromatic $K_k$, there are monochromatic subgraphs with chromatic number larger than $k$, and thus it is only through the clique number of $H$ that we can ensure that there are no monochromatic copies of $H$ in the first $q-1$ colors.}. Moreover, if $H$ is connected, then every copy of $H$ in color $q$ must lie in one of the $r_{q-1}(k)-1$ parts. Therefore, if $H$ is a $t$-vertex connected graph with clique number $k$, the Ramsey-blowup coloring implies that $m_q(H,n) \leq (r_{q-1}(k)-1)^{1-t} (n)_t$ and thus that $c_q(H) \leq (r_{q-1}(k)-1)^{1-t}$. We remark that the Ramsey-blowup coloring exactly generalizes the Tur\'an coloring, because the one-color Ramsey number of $K_k$ is equal to $k$. Similarly, the following definition generalizes the notion of a bonbon.

\begin{Def}
	Let $q,k \geq 2$ be integers. We say that a $q$-coloring of $E(K_n)$ is \emph{Ramsey-blowup-like} if there is a color so that the graph of edges in that color is a disjoint union of $r_{q-1}(k)-1$ cliques, each of size $\flo{\frac{n}{r_{q-1}(k)-1}}$ or $\cei{\frac{n}{r_{q-1}(k)-1}}$, and such that there is no monochromatic $K_k$ in any of the other colors.

	Let $H$ be a $t$-vertex connected graph with clique number $k$. We say that $H$ is a \emph{$q$-color bonbon} if, for all sufficiently large $n$, any $q$-coloring of $E(K_n)$ with the minimum number of monochromatic copies of $H$ is Ramsey-blowup-like. In particular, if $H$ is a $q$-color bonbon, then $c_q(H) = (r_{q-1}(k)-1)^{1-t}$.
\end{Def}
Note that in the case $q=2$, this definition of a 2-color bonbon coincides with that of a bonbon from \cref{def:bonbon}. This definition might  not appear at first to be the natural extension to more than two colors, as we do not fully specify what happens between the parts.
In \cref{sec:weaker-bonbon-def}, we discuss why this definition seems to be the right one for this problem.

Our second main theorem is that, conditional on a natural but unproven assumption about Ramsey numbers, a graph obtained from a large clique by adding sufficiently many pendant edges is a three-color bonbon. To state this assumption, we make the following definition, where $r(a,b)$ denotes the off-diagonal two-color Ramsey number, namely the least $r$ so that any red/blue coloring of $E(K_r)$ contains a monochromatic red $K_a$ or a monochromatic blue $K_b$.
\begin{Def}\label{def:polite}
	We say that a positive integer $k$ is \emph{polite} if 
	\begin{equation}\label{eq:polite-def-1}
		r(k,\cei{k/2}) \leq 2^{-31} r(k,k)
	\end{equation}
	and
	\begin{equation}\label{eq:polite-def-2}
		\frac{r(k,k)-1}{r(k,k-1)} \geq 1+ 25 \left(\frac{r(k,\cei{k/2})}{r(k,k)}\right)^{1/4} .
	\end{equation}
\end{Def}

\begin{conj}\label{conj:assumptions}
	Every sufficiently large integer is polite.
\end{conj}
\begin{rem}
	Erd\H os \cite{MR655605} wrote in 1981 that ``Almost nothing is known about the local growth properties of $r(n,m)$'', and not much has changed about our knowledge in the intervening 40 years. Indeed, \cref{conj:assumptions} would follow from very natural conjectures about the growth of Ramsey numbers which have seen no progress in decades.

	First, since it is known that both $r(k,k)$ and $r(k,\cei{k/2})$ grow exponentially, and since it is natural to expect the latter to grow at a lower exponential rate than the former, it is certainly natural to expect that $r(k,\cei{k/2})= o(r(k,k))$, which would imply (\ref{eq:polite-def-1}) for sufficiently large $k$. Additionally, Burr and Erd\H os (see \cite{MR655605}) conjectured that $r(k,k) \geq (1+c)r(k,k-1)$ for some absolute constant $c>0$, which would imply (\ref{eq:polite-def-2}) in conjunction with $r(k,\cei{k/2})= o(r(k,k))$. However, \cref{conj:assumptions} seems currently out of reach.
\end{rem}

With this definition and conjecture, we can now state our main three-color theorem.
\begin{thm}\label{thm:3-color}
	Let $k$ be a sufficiently large polite integer, and let $t$ be sufficiently large with respect to $k$. A graph $H$ obtained from $K_k$ by adding $t-k$ pendant edges is a 3-color bonbon.
\end{thm}

Prior to this result, there was only one non-$3$-common graph whose three-color Ramsey multiplicity constant was known exactly, namely the triangle $K_3$. In \cite{MR3071377}, it is proved that $c_3(K_3)= \frac{1}{25}$, and moreover, there is a complete characterization of the extremal colorings. Among them is the Ramsey-blowup coloring, but there are also other extremal colorings which are not Ramsey-blowup-like. Because of this, $K_3$ is not a three-color bonbon. Assuming \cref{conj:assumptions}, \cref{thm:3-color} yields an infinite family of three-color bonbons, and in particular yields an infinite family of new examples of  graphs which are not $3$-common and whose three-color Ramsey multiplicity constant is known exactly.

We remark that much of the proof of \cref{thm:3-color} mimics that of \cref{thm:main}, but several additional complications arise in the three-color setting, and additional ideas are needed to overcome them. Moreover, these complications seem to be inherent to the problem. Indeed, as discussed in \cref{sec:more-colors}, there are major obstructions to extending our proofs to four or more colors, and in fact, there is some reason to believe that no such result is true when the number of colors is at least five: perhaps there do not exist \emph{any} $q$-color bonbons for $q \geq 5$.

\subsection{Ramsey multiplicity upon edge deletion}\label{subsec:edge-deletion}
As an application of \cref{thm:main}, we are able to resolve a question of Huang about how edge-deletion affects the Ramsey multiplicity constant. Namely, recall that the Burr--Rosta conjecture asserts that $c(H) = 2^{1-m}$ for any graph $H$ with $m$ edges. If the Burr--Rosta conjecture were true, it would imply that if $H'$ is obtained from $H$ by deleting a single edge, then $c(H') = 2c(H)$.

The Burr--Rosta conjecture is false, but it is natural to wonder whether a weakening of the equality $c(H')=2c(H)$ is nonetheless true. Note that the Ramsey multiplicity constant is a monotone parameter, so certainly $c(H') \geq c(H)$ if $H'$ is obtained from $H$ by deleting an edge. Huang (private communication) asked whether there is an absolute constant $C>0$ so that $c(H') \leq C \cdot c(H)$ for all graphs $H$ and all subgraphs $H'$ obtained by deleting an edge. \cref{thm:main} implies that this is false, in a strong form. Namely, let $H$ be obtained from $H_0=K_5$ by adding $t-5$ pendant edges. If $t$ is sufficiently large, then \cref{thm:main} implies that $c(H) = 4^{1-t}$. Now, let $H'$ be obtained from $H$ by deleting one of the edges of the $K_5$. Note that $K_5 \setminus e$ is a $4$-critical graph, so we may again apply \cref{thm:main} to conclude that $c(H') = 3^{1-t}$ if $t$ is sufficiently large. This shows that $c(H')/c(H)$ cannot be upper-bounded by a constant, and in fact may be exponentially large in the number of vertices of $H$. We remark that similar questions about the ordinary Ramsey number, rather than the Ramsey multiplicity constant, have been recently studied in \cite{2208.11181}.

\subsection{Outline and notation}

The rest of this paper is organized as follows. In \cref{sec:general-lemmas}, we prove some general lemmas about Ramsey multiplicity that we will need. 
In \cref{sec:proof-main}, we present the proof of \cref{thm:main}, which is organized as a series of claims which repeatedly refine the structure of a two-coloring minimizing the number of monochromatic copies of $H$, to eventually conclude that the coloring is the Tur\'an coloring. In \cref{sec:three-colors}, we prove \cref{thm:3-color}; the proof bears many similarities to that of \cref{thm:main}, so we shorten or omit several of the proofs, choosing to focus on the places where new ideas are needed to handle the added complexity of the three-color case. Finally, we end with some concluding remarks: \cref{sec:weaker-bonbon-def} discusses why we define $q$-color bonbons for $q \geq 3$ as we do, \cref{sec:more-colors} discusses why our techniques seem to fail for $q>3$ (and why $q$-color bonbons may not even exist for large $q$), and \cref{sec:open-problems} lists some further open problems that arise from this work.

In the two-color case, we use the colors red and blue, which we denote by $R$ and $B$. In the three-color case, we add the color yellow, denoted by $Y$. For clarity of presentation, we systematically omit floor and ceiling signs whenever they are not crucial.

\section{General lemmas about Ramsey multiplicity}\label{sec:general-lemmas}
In this section we collect four general results which we use in our proofs.
Our first lemma says that in a coloring minimizing the number of monochromatic copies of some graph $H$, all vertices must lie in roughly the same number of monochromatic copies. For a graph $H$, a two-coloring $\chi$ of $E(K_n)$, and a vertex $v \in V(K_n)$, let us denote by $m_v(H,\chi)$ the number of monochromatic labeled copies of $H$ that contain $v$.

\begin{lem}\label{lem:degree-regular}
	For every graph $H$ on $t$ vertices, there exists a constant $C>0$ such that the following holds. For any two-coloring $\chi$ of $E(K_n)$ with the minimum number of monochromatic labeled copies of $H$, and for any $v \in V(K_n)$,
	\[
		\left( 1- \frac Cn \right) \frac tn m(H,n) \leq m_v(H,\chi)\leq \left( 1+\frac Cn \right) \frac tn m(H,n).
	\]
\end{lem}
\begin{proof}
	By picking the constant $C$ appropriately, we may assume that $n$ is sufficiently large in terms of $H$.
	For simplicity, we abbreviate $m_v(H,\chi)$ as simply $m_v$. Let $u$ and $w$ be two vertices of $K_n$ such that $m_u$ is minimal and $m_w$ is maximal among all vertices of $K_n$. We have that
	\[
		t \cdot m(H,n) = \sum_{v \in V(K_n)} m_v,
	\]
	which implies that $m_u \leq \frac{t}{n} m(H,n) \leq m_w$. 

	We claim that for some constant $C>0$ depending only on $H$, we have that $m_u \geq (1-C/n) m_w$. This immediately implies the desired result (up to changing the constant $C$), since the fact that the maximum and minimum values of $m_v$ differ by at most a factor of $1-C/n$ implies that every value of $m_v$ is within a factor of $1\pm 2C/n$ of the average value. 

	So suppose for contradiction that $m_u < (1-C/n) m_w$. Consider the coloring $\chi'$ obtained by deleting $w$ and replacing it with a clone $u'$ of $u$, namely setting $\chi'(u',v) = \chi(u,v)$ for any $v \notin \{u,u'\}$. We also color the edge $uu'$ red. Then we claim that $\chi'$ has strictly fewer monochromatic labeled copies of $H$ than $\chi$ does, contradicting our assumption on $\chi$. Indeed, in deleting $w$, we destroy $m_w$ monochromatic labeled copies of $H$, and when making the clone $u'$, we add $m_u$ monochromatic labeled copies of $H$ containing $u'$ but not $u$. Additionally, there are at most $(n-2)^{t-2}$ monochromatic labeled copies of $H$ containing both $u$ and $u'$. In all, we find that
	\[
		m(H,\chi') \leq m(H,\chi) - m_w + m_u + ({n-2})^{t-2}< m(H,\chi) + n^{t-2} - \frac{Cm_w}{n}.
	\]
	Since $n$ is sufficiently large, we may assume that $m(H,n) \geq \frac{c(H)}{2}n^t$. Therefore,
	\[
		m_w \geq \frac tn m(H,n) \geq \frac tn \frac{c(H)}{2}n^t = \frac{tc(H)}{2} n^{t-1}.
	\]
	Thus, if we let $C = \frac{2}{tc(H)}$ be a constant depending only on $H$, then we see that
	\[
		\frac{Cm_w}{n} \geq n^{t-2}
	\]
	implying that $m(H,\chi')<m(H,\chi)$, as claimed. Thus, we have our contradiction, and find that $m_u \geq (1-C/n) m_w$. 
\end{proof}

We will need the following lemma (which essentially appears in \cite{MR151956}), a simple and well-known lower bound on the Ramsey multiplicity constants of cliques. We remark that a better lower bound is known \cite{MR2927637} for $c(K_k)$ for sufficiently large $k$, but we stick with the following since it applies for all $k$.
\begin{lem}\label{lem:clique-multiplicity}
	Let $H_1,H_2$ be graphs on $h_1,h_2$ vertices, respectively, and suppose that $h_1,h_2 \leq h$. For any $n \geq 4^{h}$, any two-coloring of $E(K_n)$ contains at least $4^{-h^2}(n)_{h_1}$ monochromatic red labeled copies of $H_1$, or at least $4^{-h^2}(n)_{h_2}$ monochromatic blue labeled copies of $H_2$.
\end{lem}
\begin{proof}
	Let $r=r({h_1},{h_2})$, and fix a two-coloring of $E(K_n)$ for $n \geq 4^h$. The Erd\H os--Szekeres \cite{MR1556929} bound implies that $r \leq \binom{h_1+h_2-2}{h_1-1}\leq 4^{h-1} \leq n$. By the definition of $r$, every $r$-subset of $V(K_n)$ contains a red $K_{h_1}$ or a blue $K_{h_2}$. Suppose first that at least half of these subsets contain a red $K_{h_1}$. We have that every red copy of $K_{h_1}$ appears in exactly $\binom{n-h_1}{r-h_1}$ subsets of size $r$. Therefore, by double-counting, the total number of red $K_{h_1}$ is at least
	\[
		\frac{\frac 12\binom nr}{\binom{n-h_1}{r-h_1}} = \frac{\binom n{h_1}}{2\binom r{h_1}} \geq \frac{r^{-h_1}}2 \binom n{h_1} \geq 4^{-h^2} \binom n{h_1},
	\]
	using the bounds $\binom r{h_1} \leq r^{h_1}$, $h_1 \leq h$ and $r \leq 4^{h-1}$. Since every red $K_{h_1}$ contains exactly $h_1!$ red labeled copies of $H_1$, we find at least $4^{-h^2}(n)_{h_1}$ red labeled copies of $H_1$ in this case. By interchanging the roles of $h_1$ and $h_2$ and blue and red, we get the other desired conclusion in case at least half of the $r$-subsets of $V(K_n)$ contain a blue $K_{h_2}$, which completes the proof.
\end{proof}

The next result we need is due to He and the authors \cite[Theorem 3.5]{2109.09205}, though a similar result was proved earlier by Bollob\'as and Nikiforov \cite[Theorem 9]{MR2370517}. It can be viewed as a combination of the stability and supersaturation versions of Tur\'an's theorem, and says that if a graph $G$ has minimum degree close to the Tur\'an threshold for containing a copy of $K_k$, while also containing few copies of $K_k$, then it must be close to $(k-1)$-partite. For more on this intuition and motivation, see \cite[Section 3]{2109.09205}.

\begin{lem}[{\cite[Theorem 3.5]{2109.09205}}]\label{lem:stability}
	For every $\varepsilon>0$ and every integer $k \geq 3$, there exist $\alpha,\delta>0$ such that the following holds for all $n$. Suppose $G$ is a graph on $n$ vertices with minimum degree at least $(1- \frac{1}{k-1}- \delta)n$ and with at most $\alpha \binom nk$ copies of $K_k$. Then $V(G)$ can be partitioned into $V_1 \sqcup \dotsb \sqcup V_{k-1}$, such that the total number of internal edges in $V_1,\dotsc,V_{k-1}$ is at most $\varepsilon \binom n2$. 

	Moreover, we may take $\delta= \min \{1/(2k^2),\varepsilon/2\}$ and $\alpha = k^{-10k} \varepsilon$.
\end{lem}

Finally, we will need the supersaturation version of the Erd\H os--Stone--Simonovits theorem, due to Erd\H os and Simonovits \cite{MR726456}; the quantitative estimate we state follows from~\cite{MR609100} and a standard proof of the supersaturation theorem, e.g.\ \cite[Lemma 2.1]{MR2866732}.
\begin{lem}\label{lem:supersaturation}
	Let $H_0$ be an $h$-vertex graph of chromatic number $k$. For every $\delta>0$, there exists some $\gamma = \gamma(H_0,\delta)>0$ such that the following holds for sufficiently large $n$. If $G$ is an $n$-vertex graph with at least $(1- \frac{1}{k-1}+\delta)\binom n2$ edges, then $G$ contains at least $\gamma (n)_h$ labeled copies of $H_0$. Moreover, we may take $\gamma = \delta^{1000 h^2}$.
\end{lem}

\section{Two-color bonbons}\label{sec:proof-main}
In this section we prove \cref{thm:main}. From now on, we fix $k \geq 4$ and let $H_0$ be a fixed $k$-critical graph on $h$ vertices, and we also treat $t \geq (1000hk)^{10} h^{10k}$ as fixed.
\subsection{Choices of parameters}
For convenience, we record here all the parameters we will need in our proof of \cref{thm:main}. Somewhat unusually, the constraints we have on our parameters do not naturally form a linear order (i.e.\ it is not the case that we have parameters $a_1,\dots,a_m$ such that each $a_i$ must be sufficiently small with respect to $a_{i+1}$), but rather they form a poset structure, generated by the following inequalities:
\[
	\theta \ll \frac 1k, \qquad \varepsilon \ll \theta, \qquad \varepsilon \ll \frac 1h, \qquad \lambda \ll \varepsilon, \qquad \gamma \ll \lambda.
\]
In addition, there is one further parameter $\tau$, which is defined as $(1+\lambda)^{-t}$. In contrast to the other parameters, we cannot simply pick $\tau$ to be sufficiently small, as we end up needing both upper and lower bounds on $\tau$. Hence, it is defined as $(1+\lambda)^{-t}$ in order to be sufficiently large with respect to $t$, and then our lower bound on $t$ ensures that $\tau \ll \gamma$.

The rest of this subsection formalizes the choices of parameters and the inequalities we need them to obey; the reader may safely skip this subsection and simply bear the above qualitative dependencies in mind.

We first pick $\theta,\varepsilon,$ and $\lambda$ to be polynomially small in $k$ and $h$, defined as
\begin{align*}
	\theta &= \frac{1}{50k}&
	\varepsilon &= \frac{\theta^2}{2h^2k^2}=\frac{1}{5000h^2k^4}&
	\lambda &= \frac{\varepsilon^2}{(200h)^2} = \frac{1}{2\cdot 10^8 h^4 k^4}.
\end{align*}
Additionally, we let $\gamma$ be the parameter from \cref{lem:supersaturation}, applied to the graph $H_0$ and $\delta=\lambda$, namely 
$$\gamma = \lambda^{1000h^2}.$$
We now define a further parameter depending on both $k$ and $t$, namely
\begin{align*}
	\tau &= \frac 1{(1+\lambda)^{t}}.
\end{align*}
The following lemma records the important inequalities that we will need these parameters to satisfy. It is straightforward to check that the choices above guarantee that all these inequalities are satisfied.

\begin{lem}
	The following inequalities hold.
	\begin{lemenum}
		\item $\tau \leq \gamma \lambda^h \leq 4^{-h^2-1} \lambda^h$. \label{lemitem:tau-v-gamma}
		\item $2 \tau^{h^{-k}} \leq k^{-10k}\varepsilon$. \label{lemitem:tau-small-stability}
		\item $\theta \geq 2k \tau^{1/h}$. \label{lemitem:theta-2ktau}
		\item $e^{8(k-1)\sqrt \lambda t} > 4^{h^2} t$.\label{lemitem:8k-1}
		\item $\frac{1}{k-1}-\sqrt{2 \varepsilon} \geq \frac 1k$. \label{lemitem:equitable-error}
	\end{lemenum}
\end{lem}
\begin{proof}
	\hfill
	\begin{enumerate}[label=(\alph*)]
		\item \label{itemproof:tau-ub} We know that $\gamma = \lambda^{1000h^2} \leq 4^{-h^2-1}$, so the second inequality is immediate. For the first, we have that $\gamma \lambda^h \geq \lambda^{5000h^2} \geq 2^{-5000 h^2/\lambda}= 2^{-10^{12} h^6 k^4}$, using the fact that $x \geq 2^{-1/x}$ for all $x>0$. Since $1+x \geq 2^x$ for all $x \in [0,1]$, we also see that $\tau \leq 2^{- \lambda t}$. Finally, $t \geq (1000 hk)^{10}$, so $\lambda t \geq 10^{12} h^6 k^4$, which proves the claim.
		\item Since $h \geq k \geq 4$, it suffices to prove that $\tau \leq 2^{-h^{10k}}\varepsilon$, which itself follows from $\tau \leq 2^{-h^{10k}/\varepsilon}$. We saw in part \ref{itemproof:tau-ub} that $\tau \leq 2^{-\lambda t}$, so it suffices to prove that $\lambda t\geq h^{10k}/\varepsilon$, which is true since $\lambda(1000 hk)^{10} \geq 10^{12}h^6 k^4 \geq 1/\varepsilon$.
		\item From part \ref{itemproof:tau-ub}, we see that $\tau^{1/h} \leq \lambda$, which yields the desired bound since $2k \lambda \leq \theta$.
		\item We first note that by the computation in part \ref{itemproof:tau-ub}, we have that $\sqrt{\lambda t} \geq 10^6 h^3 k^2>h^2$. Since $k \geq 3$, we see that $e^{8(k-1)} \geq 4$. Therefore, $$e^{8(k-1)\sqrt \lambda t} \geq (4^{\sqrt t})^{\sqrt{\lambda t}} \geq (4t)^{\sqrt{\lambda t}} \geq t \cdot 4^{\sqrt{\lambda t}} > 4^{h^2}t,$$ where the second inequality uses that $4^{\sqrt t} \geq 4t$ for all $t \geq 4$. 
		\item We have that $\frac 1{k-1} - \frac 1k = \frac{1}{k(k-1)} \geq \frac{1}{k^2}$. Since $\sqrt{2 \varepsilon} < 1/k^3$, this proves the desired bound.
	\qedhere\end{enumerate}
\end{proof}

\subsection{Proof of Theorem \ref{thm:main}}
Recall that we have fixed $H_0$ and $t$. We let $H$ be obtained from $H_0$ by arbitrarily adding $t-h$ pendant edges to $H_0$, so that $H$ has $t$ vertices. Let the vertices of $H_0$ be $u_1,\dots,u_h$, and let $s_i$ be the number of pendant edges incident to $u_i$ for $1 \leq i \leq h$, so that $s_1+\dotsb+s_h = t-h$. We assume without loss of generality that $s_1 \geq \dotsb \geq s_h$, which in particular implies that $s_1 \geq \frac{t-h}{h}\geq \frac{t}{2h}$.

We recall (\ref{eq:turan-coloring-ub-mhn}), which gives an upper bound on $m(H,n)$ coming from the Tur\'an coloring. Namely, we have that
\begin{equation}\label{eq:turan-ub}
	m(H,n) \leq (k-1)^{1-t} (n)_t.
\end{equation}
We also henceforth let $n$ be sufficiently large in terms of $H_0$ and $t$, and let $\chi$ be an optimal two-coloring of $E(K_n)$, that is a coloring with exactly $m(H,n)$ monochromatic labeled copies of $H$. Our goal is to show that $\chi$ is isomorphic to the Tur\'an coloring. By combining \cref{lem:degree-regular} and \eqref{eq:turan-ub}, we find that for every vertex $v \in V(K_n)$, we have that
\begin{equation}\label{eq:local-turan-ub}
	m_v(H,\chi) \leq \left(1+\frac Cn\right) \frac tn m(H,n) \leq \left(1+\frac Cn\right) t(k-1)^{1-t}(n-1)_{t-1}
\end{equation}
for some constant $C$ depending only on $H$.

Let $d=(1+\lambda)n/(k-1)$. Note that $d$ is slightly larger than the red degree of any vertex in the Tur\'an coloring. The following lemma is used several times in the proof of \cref{thm:main}. Additionally, its proof exemplifies the kinds of arguments that arise throughout. 

\begin{lem}\label{lem:few-H0-in-high-deg}
Suppose $n$ is sufficiently large. Let $S\subseteq V(K_n)$ be a set of vertices, each with at least $d$ blue neighbors. Then $S$ contains fewer than $\tau (n)_h$ labeled blue copies of $H_0$.
\end{lem}

\begin{proof}
	Any labeled blue copy of $H_0$ in $S$ extends to at least $(d-h)_{t-h}$ labeled blue copies of $H$, since every vertex in this $H_0$ has blue degree at least $d$. Therefore, if there are $\tau(n)_h$ labeled blue copies of $H_0$ in $S$, then the total number of blue labeled copies of $H$ in $\chi$ is at least
	\begin{align*}
		\tau(n)_h \cdot (d-h)_{t-h} &= \tau (n)_h \cdot (1-o(1))\left(\frac{1+\lambda}{k-1}\right)^{t-h} (n)_{t-h} \geq (1-o(1)) \tau \left(\frac{1+\lambda}{k-1}\right)^{t-h} (n)_t,
	\end{align*}
	where the $o(1)$ tends to $0$ as $n \to \infty$. We claim that this is more than the upper bound in (\ref{eq:turan-ub}), whence a contradiction. To see this, we note that
	\begin{align*}
		\left( \frac{1+\lambda}{k-1} \right) ^{t-h} &= \left( \frac{k-1}{1+\lambda} \right) ^{h-1}\left( \frac{1+\lambda}{k-1} \right) ^{t-1} \geq 8 (1+\lambda)^{t-1} (k-1)^{1-t} \geq 4(1+\lambda)^t (k-1)^{1-t},
	\end{align*}
	using the fact that $h \geq k \geq 4$ and $\lambda \leq \frac 12$, and hence $(k-1)/(1+\lambda) \geq 2$. Since $(1+\lambda)^t \tau= 1$, this shows that for sufficiently large $n$, the total number of labeled blue copies of $H$ is at least
	\begin{equation}\label{eq:basic-d-lb}
		\tau(n)_h \cdot (d-h)_{t-h} \geq 2 (k-1)^{1-t} (n)_t,
	\end{equation}
	contradicting (\ref{eq:turan-ub}).
\end{proof}

We are now ready to begin the proof of \cref{thm:main}. The proof is somewhat long, since it proceeds by iteratively improving our understanding of the structure of a coloring which minimizes the number of monochromatic copies of $H$, until we eventually can conclude that such a coloring is isomorphic to the Tur\'an coloring. In order to keep the logical flow manageable, we split the proof into a number of claims, each of which provides more structural information on the coloring.

\begin{proof}[Proof of \cref{thm:main}]

Fix a red/blue coloring $\chi$ of $E(K_n)$ with the minimum number of monochromatic copies of $H$. We partition the vertices of $K_n$ into three subsets. The first, $V_R$, consists of all vertices with red degree at least $n-d$. Similarly, $V_B$ consists of those vertices with blue degree at least $n-d$. Finally, $V_{RB}$ consists of all remaining vertices, namely those vertices with both red and blue degree at least\footnote{Recall that we are omitting all floor and ceiling signs, and thus are treating $d$ as an integer even though $(1+\lambda)n/(k-1)$ need not be an integer. As such, every vertex has either red degree at least $n-d$ or blue degree at least $d$, whereas we could have an off-by-one error here if $d$ were not an integer. Such rounding and off-by-one issues occur throughout the paper, but we will not belabor this point further.} $d$. We remark that it is at this partitioning step that our proof fails for $k=3$. For indeed, if $k=3$, then $d = \frac{1+\lambda}{2}n$, meaning that $V_R$ and $V_B$ will not in general be disjoint. However, since $k \geq 4$ and $\lambda<1/2$, we have that $V_R \cap V_B=\varnothing$. 

Our first claim shows that $V_{RB}$ must be small.
\begin{claim}\label{lem:V0-small}
	$\ab {V_{RB}} \leq \lambda n$.
\end{claim}
\begin{proof}
	Suppose for contradiction that $\ab{V_{RB}} > \lambda n$. Since $n$ is sufficiently large, we may therefore assume that $\ab{V_{RB}} \geq 4^{h}$, so we may apply \cref{lem:clique-multiplicity} to the induced coloring on $V_{RB}$, with $H_1=H_2=H_0$. We then conclude without loss of generality that $V_{RB}$ contains at least $4^{-h^2}(n)_h$ labeled blue copies of $H_0$. On the other hand, since every vertex in $V_{RB}$ has blue degree at least $d$, we may apply \cref{lem:few-H0-in-high-deg} and conclude that $V_{RB}$ contains at most $\tau (n)_h$ blue $H_0$. Combining these two bounds, we see that
	\[
		\tau (n)_h \geq 4^{-h^2} ({\ab{V_{RB}}})_h \geq  (1-o(1))4^{-h^2} \lambda^h (n)_h.
	\]
	However, \cref{lemitem:tau-v-gamma} shows that $\tau \leq 4^{-h^2-1} \lambda^h$, which is a contradiction for $n$ sufficiently large.
\end{proof}

Similarly, our next claim shows that one of $V_R$ and $V_B$ must also be small.
\begin{claim}\label{lem:V_R-or-V_B-small}
	$\min \{\ab{V_R},\ab{V_B}\} \leq 18 \lambda n$.
\end{claim}
\begin{proof}
	Recall that since $k \geq 4$, we have that $V_R$ and $V_B$ are disjoint. Suppose for contradiction that $\ab{V_R},\ab{V_B} \geq 18\lambda n$. Consider the set of edges between $V_R$ and $V_B$, and suppose without loss of generality that at most half these edges are blue. Note that since every vertex in $V_B$ has blue degree at least $n-d \geq d$, we must have at most $\tau(n)_h \leq \tau \lambda^{-h}(\ab{V_B})_h$ labeled blue copies of $H_0$ in $V_B$, by \cref{lem:few-H0-in-high-deg}. On the other hand, if the average blue degree inside $V_B$ is at least $(1- \frac{1}{k-1}+\lambda)\ab{V_B}$, then \cref{lem:supersaturation} implies that $V_B$ contains at least $\gamma (\ab{V_B})_h$ labeled blue copies of $H_0$. But \cref{lemitem:tau-v-gamma} shows that $\tau < \gamma \lambda^h$, and we conclude that the average blue degree in $V_B$ is less than $(1- \frac{1}{k-1}+\lambda)\ab{V_B}$.

	Now, we add up the blue degrees of all vertices in $V_B$. On the one hand, this is at least $(n-d)\ab{V_B}$, since every vertex in $V_B$ has blue degree at least $n-d$. On the other hand, in this sum, we count every blue edge in $V_B$ twice and every blue edge between $V_B$ and $V_R \cup V_{RB}$ once. Since we assumed that at most half the edges in $V_B \times V_R$ are blue, we conclude that
	\begin{align*}
		(n-d)\ab{V_B} &\leq \left(1- \frac{1}{k-1}+\lambda\right)\ab{V_B}^2 + \frac 12 \ab{V_B} \ab{V_R} + \ab{V_B} \ab{V_{RB}}\\
		&=\left[\left(1- \frac{1}{k-1}+\lambda\right)\ab{V_B} + \frac 12 \ab{V_R} + \ab{V_{RB}}\right]\ab{V_B}\\
		&=\left(n - \left(\frac{1}{k-1}- \lambda\right)\ab{V_B} - \frac12 \ab{V_R}\right)\ab{V_B},
	\end{align*}
	where in the last step we use that $n=\ab{V_R}+\ab{V_B}+\ab{V_{RB}}$.
	This implies that
	\[
		d\geq \left(\frac{1}{k-1}- \lambda\right)\ab{V_B} + \frac 12 \ab{V_R} = \frac{1}{k-1}n - \lambda \ab{V_B} - \frac{1}{k-1}\ab{V_{RB}}+\left(\frac 12 - \frac{1}{k-1}\right)\ab{V_R}.
	\]
	Since $d = \frac{1+\lambda}{k-1}n \leq \frac{1}{k-1}n + \lambda n$, we conclude that
	\[
		\left(\frac12 - \frac{1}{k-1}\right)\ab{V_R} \leq \lambda n + \lambda \ab{V_B} + \frac{1}{k-1}\ab{V_{RB}} \leq 3 \lambda n.
	\]
	Finally, since $\frac 12 - \frac{1}{k-1} \geq \frac 16$ for all $k \geq 4$, we conclude that $\ab{V_R} \leq 18 \lambda n$, as claimed.
\end{proof}
We henceforth assume without loss of generality that $\ab{V_R} \leq 18 \lambda n$. Combining this with \cref{lem:V0-small}, we conclude that $\ab{V_B} \geq (1- 19 \lambda) n$. 
Our next claim shows that within $V_B$, few vertices have blue degree much larger than $n-d$.
\begin{claim}\label{claim:V_B-few-high-deg}
	Let $V_B'\subseteq V_B$ be the set of vertices in $V_B$ with blue degree at least $(1- \frac{1}{k-1}+3\sqrt \lambda)n$. Then $\ab{V_B'}\leq 21 \sqrt \lambda \ab{V_B}$.
\end{claim}
\begin{proof}
	We first recall from the proof of \cref{lem:V_R-or-V_B-small} that the average blue degree in $V_B$ is at most $(1- \frac{1}{k-1}+\lambda)\ab{V_B}$; indeed, if this were not the case, then \cref{lem:supersaturation} would yield many blue copies of $H_0$ in $V_B$, contradicting the fact that there cannot be many such copies since every vertex in $V_B$ has blue degree at least $n-d \geq d$.

	Additionally, we note that since $\ab{V_B} \geq (1-19 \lambda)n$, every vertex in $V_B$ has at least $n-d - 19 \lambda n$ blue neighbors in $V_B$, and
	\[
		n - d - 19 \lambda n = \left(1 - \frac{1+\lambda}{k-1}-19 \lambda\right)n \geq \left(1- \frac{1}{k-1}-20 \lambda\right)n.
	\]
	Similarly, any vertex in $V_B'$ has at least $(1 - \frac{1}{k-1}+3\sqrt \lambda)n-19 \lambda n\geq (1- \frac{1}{k-1}+2\sqrt \lambda)n$ blue neighbors in $V_B$, since $20 \lambda \leq \sqrt \lambda$.
	Now, we sum up over all $v \in V_B$ the number of blue neighbors of $v$ in $V_B$. On the one hand, this sum is at most $(1- \frac{1}{k-1}+\lambda)\ab{V_B}^2 \leq (1- \frac{1}{k-1}+\lambda)n \ab{V_B}$, by our bound on the average degree in $V_B$. On the other hand, this sum is at least
	\begin{gather*}
		\sum_{v \in V_B'} \left(1 - \frac{1}{k-1}+2\sqrt \lambda\right)n +\sum_{v \in V_B \setminus V_B'} \left(1 - \frac{1}{k-1}-20 \lambda\right)n \\
		= \left(2\sqrt \lambda - 20 \lambda\right)n \ab{V_B'} + \left(1- \frac{1}{k-1} - 20 \lambda\right)n \ab{V_B}.
	\end{gather*}
	Combining this with our upper bound, we find that
	\[
		\left(2\sqrt \lambda - 20 \lambda\right) \ab{V_B'} \leq 21 \lambda \ab{V_B}.
	\]
	Since $20 \lambda \leq \sqrt \lambda$, we conclude that $\ab{V_B'} \leq 21 \sqrt \lambda \ab{V_B}$.
\end{proof}
Our next claim shows that in fact, there are no vertices of high red degree.
\begin{claim}\label{lem:every-vertex-high-blue-deg}
	Every vertex of $K_n$ has red degree at most $\frac{1}{k-1}n+65h\sqrt \lambda n$.
\end{claim}

\begin{proof}
	Suppose for contradiction that the red degree of some vertex $v$ is at least $\frac{1}{k-1}n+65h\sqrt \lambda n$. Let $T = N_R(v) \cap V_B$. Since there are at most $19 \lambda n$ vertices outside of $V_B$, we see that $\ab T\geq \frac 1{k-1} n + 64h \sqrt \lambda n$. Recall from \cref{claim:V_B-few-high-deg} that $V_B'$ consists of those vertices in $V_B$ with blue degree at least $(1- \frac{1}{k-1}+3 \sqrt \lambda)n$, and that $\ab{V_B'} \leq 21 \sqrt \lambda \ab{V_B} \leq 21 \sqrt \lambda n$. Thus, if we let $S = T \setminus V_B'$, we see that $\ab S \geq \frac{1}{k-1}n + 64h\sqrt \lambda n - 21 \sqrt \lambda n \geq \frac{1}{k-1}n + 50 \sqrt \lambda n$, using the fact that $h \geq k \geq 4$.

	Recall that $u_1,\dots,u_h$ are the vertices of $H_0$, and that $u_i$ is incident to $s_i$ pendant edges, with $s_1 \geq \dotsb \geq s_h$. Let $H_0' = H_0 \setminus \{u_1\}$. We apply \cref{lem:clique-multiplicity} to $S$, with $H_1 = H_0'$ and $H_2 = H_0$. We conclude that $S$ contains at least $4^{-h^2}(\ab S)_h$ labeled blue copies of $H_0$, or at least $4^{-h^2}(\ab S)_{h-1}$ labeled red copies of $H_0'$. If the former happens, then the total number of labeled blue copies of $H_0$ in $S$ is at least
	\[
		4^{-h^2}(\ab S)_h \geq 4^{-h^2} (k-1)^{-h} (n)_h > \tau (n)_h,
	\]
	a contradiction to \cref{lem:few-H0-in-high-deg}, since every vertex in $S$ has blue degree at least $n-d \geq d$.

	Therefore, we may assume that $S$ contains at least $4^{-h^2}(\ab S)_{h-1}$ labeled red copies of $H_0'$. Recall that since $S \subseteq N_R(v)$, every red copy of $H_0'$ in $S$ yields a red copy of $H_0$ containing $v$. Note that since $S \subseteq V_B \setminus V_B'$, every vertex in $S$ has red degree at least $(\frac{1}{k-1}-3 \sqrt \lambda)n$. This implies that given a labeled red copy of $H_0'$ in $S$, we may extend it to a labeled red copy of $H$ in at least
	\[
		\left(\left(\frac{1}{k-1}+65h\sqrt \lambda\right)n-t\right)_{s_1} \left(\left(\frac{1}{k-1}-3\sqrt \lambda\right)n-t\right)_{s_2}\dotsb \left(\left(\frac{1}{k-1}-3\sqrt \lambda\right)n-t\right)_{s_h}
	\]
	ways. In this count, we first choose $s_1$ distinct red neighbors of $v$, then $s_2,\dots,s_h$ distinct red neighbors of the $h-1$ vertices of the fixed copy of $H_0'$ in $S$. By subtracting $t$ from each term, we can ensure that all these chosen vertices are distinct, and thus that we are truly embedding a red labeled copy of $H$. This quantity is at least
	\begin{align*}
		\bigg(\bigg(\frac{1}{k-1}&+64h\sqrt \lambda\bigg)n\bigg)_{s_1} \bigg(\bigg(\frac{1}{k-1}-4\sqrt \lambda\bigg)n\bigg)_{s_2+\dotsb+s_h} \\
		&=(1-o(1))\left(\frac{1}{k-1}+64 h\sqrt \lambda\right)^{s_1} \left(\frac{1}{k-1}-4 \sqrt \lambda\right)^{s_2+\dotsb+s_h} (n)_{t-h}\\
		&=(1-o(1)) (k-1)^{h-t} (n)_{t-h} (1+64 h(k-1)\sqrt \lambda)^{s_1}(1-4(k-1)\sqrt \lambda)^{s_2+\dotsb+s_h}\\
		&\geq (1-o(1)) (k-1)^{h-t} (n)_{t-h} \exp\left(32h(k-1)\sqrt \lambda s_1 - 8(k-1)\sqrt \lambda(s_2+\dotsb+s_h)\right),
	\end{align*}
	where we use the inequalities $1+x \geq e^{x/2}, 1-x \geq e^{-2x}$, valid for all $x \in [0,\frac 12]$ (as well as the fact that $32 hk\sqrt \lambda \leq \frac 12$, so that we may apply these bounds).
	Now, we note that $hs_1 \geq s_1+\dotsb+s_h=t-h$, which implies that $4h s_1 -(s_2+\dotsb+s_h) \geq 3hs_1 \geq t$. Putting this all together, we find that any labeled red copy of $H_0'$ in $S$ extends to at least
	\[
		(1-o(1))(k-1)^{h-t} e^{8(k-1)\sqrt \lambda t} (n)_{t-h}
	\]
	labeled red copies of $H$, where the $o(1)$ term tends to $0$ as $n \to \infty$. Recall from above that $S$ contains at least $4^{-h^2}(\ab S)_{h-1}$ labeled red copies of $H_0'$. So the total number of labeled red copies of $H$ containing $v$ is, for $n$ sufficiently large, at least
	\[
		4^{-h^2} (k-1)^{1-h} (n)_{h-1} \cdot (k-1)^{h-t} e^{8(k-1)\sqrt \lambda t} (n)_{t-h} \geq 4^{-h^2} e^{8(k-1)\sqrt \lambda t} (k-1)^{1-t} (n)_{t-1}.
	\]
	But $e^{8(k-1)\sqrt \lambda t} >4^{h^2}t$ by \cref{lemitem:8k-1}, which yields a contradiction to the bound (\ref{eq:local-turan-ub}). Thus, there is no vertex $v$ of red degree at least $(\frac{1}{k-1}+65 h \sqrt \lambda)n$.
\end{proof}

\cref{lem:every-vertex-high-blue-deg} implies that the blue graph on $K_n$ has minimum degree at least $(1- \frac{1}{k-1}-65h\sqrt \lambda)n$.
Our next claim shows that the coloring has few blue $K_k$, which will put us into a position to apply \cref{lem:stability}.

\begin{claim}\label{claim:few-Kk}
	The number of blue $K_k$ is at most $2\tau^{h^{-k}} \binom {n}{k}$. 
\end{claim}
\begin{proof}
	By \cref{lem:every-vertex-high-blue-deg}, every vertex in $K_n$ has blue degree at least $(1- \frac{1}{k-1}-65 h \sqrt \lambda)n \geq d$, since $65h \sqrt \lambda \leq \frac 14$. This implies, by \cref{lem:few-H0-in-high-deg}, that the coloring contains fewer than $\tau (n)_h$ labeled blue copies of $H_0$. 

	Suppose for contradiction that the coloring contains at least $2 \tau^{h^{-k}} \binom nk$  blue copies of $K_k$, and let $\h$ be the $k$-uniform hypergraph on $V(K_n)$ whose edges are these copies of $K_k$. Fix a proper $k$-coloring of $H_0$, let $a_1,\dots,a_k$ be the sizes of the color classes, and let $\K$ be the complete $k$-partite $k$-uniform hypergraph with parts of sizes $a_1,\dots,a_k$. An argument of Erd\H os \cite{MR183654} (see also \cite{HypergraphSidorenko}) implies that there are at least $(1-o(1))(2 \tau^{h^{-k}})^{a_1\dotsb a_k} n^{h}\geq (2-o(1)) \tau (n)_h$ homomorphisms $\K \to \h$, since $a_1\dotsb a_k \leq h^k$. If $n$ is sufficiently large, then at least two-thirds of these homomorphisms are injective, which implies that the coloring contains at least $\tau (n)_h$ labeled blue copies of a complete $k$-partite graph with parts of sizes $a_1,\dots,a_k$. But each such labeled copy contains a unique blue copy of $H_0$ with matching labels, a contradiction.
\end{proof}

We now apply \cref{lem:stability} to the blue graph on $K_n$ with parameter $\varepsilon$. We note that we may do so, since we proved that the blue graph has minimum degree at least $(1- \frac{1}{k-1}- 65h \sqrt \lambda)n$ and at most $2\tau^{h^{-k}} \binom nk$ copies of $K_k$, and we have that $2\tau^{h^{-k}} \leq k^{-10 k} \varepsilon$ by \cref{lemitem:tau-small-stability}, and $65h\sqrt\lambda \leq \min \{\frac{1}{2k^2}, \frac \varepsilon 2\}$ by our choice of $\lambda$. \cref{lem:stability} then outputs a partition of $V(K_n)$ into $k-1$ parts $V_1,\dots,V_{k-1}$ with at most $\varepsilon \binom {n}2$ internal edges among all the parts.  Moreover, we can assume without loss of generality that this partition minimizes the number of internal blue edges among all partitions of $V(K_n)$ into $k-1$ parts; in other words, we assume that $V_1,\dots,V_{k-1}$ is a max $(k-1)$-cut of the blue graph. As a max $(k-1)$-cut, the partition has the following property: every vertex has at most as many blue neighbors in its own part as in any other part. Indeed, if this were not true, we could decrease the number of internal blue edges by moving some vertex to another part, in which it has fewer blue neighbors. 

The next claim records some further properties of this partition, namely that it is close to equitable and that most of the edges between $V_i$ and $V_j$ for $i \neq j$ are blue.

\begin{claim}\label{lem:partition-properties}
	The partition $V(K_n) = V_1 \sqcup \dotsb \sqcup V_{k-1}$ has the following properties.
	\begin{enumerate}[label=(\roman*)]
		\item For each $1 \leq i \leq k-1$, we have that
		\begin{equation}\label{eq:part-size-bounds}
			\frac{n}{k-1} - \sqrt{2\varepsilon}n \leq \ab{V_i} \leq \frac{n}{k-1} + \sqrt{2\varepsilon}n.
		\end{equation}

		\item For each $1 \leq i \neq j \leq k-1$, we have that 
		\begin{equation}\label{eq:pairwise-blue-lb}
			e_B(V_i,V_j) \geq (1- k^2 \varepsilon) \ab{V_i} \ab{V_j}.
		\end{equation}
	\end{enumerate}
\end{claim}

\begin{proof}
The blue graph has minimum degree at least $(1- \frac{1}{k-1}-65h\sqrt \lambda)n \geq (1- \frac{1}{k-1}- \varepsilon)n$, and thus there are at least $(1- \frac{1}{k-1}- \varepsilon)\frac{n^2}{2}$ blue edges. This implies that
\[
	\left(1 -\frac{1}{k-1}- \varepsilon\right) \frac{n^2}{2} \leq \sum_{i=1}^{k-1} e_B(V_i) + \sum_{1 \leq i <j\leq k-1} e_B(V_i,V_j) \leq \varepsilon \frac{n^2}{2} + \sum_{1 \leq i <j\leq k-1} e_B(V_i,V_j),
\]
since there are at most $\varepsilon \binom n2 \leq \varepsilon \frac{n^2}{2}$ internal blue edges among $V_1,\dots,V_{k-1}$. Rearranging and multiplying by $2$, we find that
\begin{align*}
	2 \varepsilon n^2 &\geq \left(1- \frac{1}{k-1}\right) n^2 - 2 \sum_{1 \leq i<j\leq k-1} e_B(V_i,V_j)\\
	&=\sum_{i=1}^{k-1} \left(\ab{V_i} - \frac{n}{k-1}\right)^2 + \sum_{1 \leq i<j\leq k-1} 2\left(\ab{V_i} \ab{V_j} - e_B(V_i,V_j)\right),
\end{align*}
using the fact that $n^2 = \sum_i \ab{V_i}^2 + 2 \sum_{i<j} \ab{V_i} \ab{V_j}$.
Since each of the summands above is non-negative, we find that $(\ab{V_i}-\frac{n}{k-1})^2 \leq 2 \varepsilon n^2$ for all $i$ and that $\ab{V_i} \ab{V_j} - e_B(V_i,V_j) \leq \varepsilon n^2$ for all $i \neq j$. The former is equivalent to (\ref{eq:part-size-bounds}). The latter implies that
\[
	e_B(V_i,V_j) \geq \ab{V_i} \ab{V_j} - \varepsilon n^2 \geq \ab{V_i} \ab{V_j} - \varepsilon k^2 \ab{V_i} \ab{V_j}=(1- k^2 \varepsilon)\ab{V_i}\ab{V_j},
\]
yielding (\ref{eq:pairwise-blue-lb}), since (\ref{eq:part-size-bounds}) implies that $\ab{V_i} \geq \frac{n}{k-1}-\sqrt{2 \varepsilon}n \geq \frac{n}{k}$ by \cref{lemitem:equitable-error}.
\end{proof}
We will proceed to study the structure of the coloring with respect of $V_1,\dotsc,V_{k-1}$, to eventually conclude that each $V_i$ is monochromatic red, and that all edges between parts are blue. To begin with, the next claim shows that no vertex in $K_n$ can have considerable blue degree to each $V_i$. 

Recall that $H_0$ is $k$-critical, meaning there is an edge $uu' \in E(H_0)$ whose deletion yields a $(k-1)$-colorable graph. Let $\wt{H_0} = H_0 \setminus \{u\}$ be obtained from $H_0$ by deleting one of the endpoints of this edge, so that $\wt{H_0}$ is $(k-1)$-colorable. Fix a proper $(k-1)$-coloring of $\wt{H_0}$, and let its color classes have sizes $b_1,\dots,b_{k-1}$.
\begin{claim}\label{lem:high-red-deg}
	Fix a vertex $v \in V(K_n)$, and let $U_i = N_B(v) \cap V_i$ denote the set of blue neighbors of $v$ inside $V_i$, for $1 \leq i \leq k-1$. Then there exists some $i$ such that $\ab{U_i} < \theta \ab{V_i}$.
\end{claim}
\begin{proof}
	Suppose for contradiction that $\ab{U_i} \geq \theta \ab{V_i}$ for all $1 \leq i \leq k-1$. By (\ref{eq:pairwise-blue-lb}), we know that $e_B(V_i,V_j) \geq (1- k^2 \varepsilon) \ab{V_i} \ab{V_j}$ for all $i \neq j$. Therefore, for all $i \neq j$, we have that
	\begin{equation}\label{eq:density-in-subsets}
		e_B(U_i,U_j) \geq \ab{U_i} \ab{U_j} - k^2 \varepsilon \ab{V_i} \ab{V_j} \geq \left( 1 - \frac{k^2 \varepsilon}{\theta^2} \right) \ab{U_i} \ab{U_j}.
	\end{equation}
	Thus, if we pick a random vertex from $U_i$ and a random vertex from $U_j$, they will be connected by a blue edge with probability at least $1- k^2 \varepsilon/\theta^2$. By the union bound, this implies that if we pick $b_i$ random vertices from $U_i$ (with replacement) for each $1 \leq i \leq k-1$, then these vertices will form a blue homomorphic image of $\wt{H_0}$ with probability at least $1 - \binom{h-1}2 k^2 \varepsilon/\theta^2 \geq 1- h^2 k^2 \varepsilon/(2\theta^2) = 3/4$, since $\varepsilon = \theta^2/(2h^2k^2)$. Additionally, if $n$ is sufficiently large, then with probability at least $3/4$, these $h-1$ vertices will all be distinct, so we will find a genuine blue copy of $\wt{H_0}$. Therefore, the number of blue labeled copies of $\wt{H_0}$ in $S \coloneqq U_1 \cup \dotsb \cup U_{k-1}$ is at least
	\begin{align*}
		\frac 12 \prod_{i=1}^{k-1} \ab{U_i}^{b_i} &\geq \frac{\theta^{h-1}}{2} \prod_{i=1}^{k-1} \ab{V_i}^{b_i} &[\ab{U_i} \geq \theta \ab{V_i}]\\
		&\geq \frac{\theta^{h-1}}{2} (1- 2k\sqrt \varepsilon)^{h-1} \left( \frac{n}{k-1} \right) ^{h-1} &[\textstyle\ab{V_i} \geq \frac{n}{k-1} - \sqrt{2 \varepsilon}n] \\
		&\geq \left( \frac{\theta}{2k} \right) ^h (n)_{h-1} &[\textstyle 2k\sqrt \varepsilon \leq \frac 12] \\
		& \geq \tau (n)_{h-1}, &[\theta \geq 2k \tau^{1/h}]
	\end{align*}
	where the final step uses \cref{lemitem:theta-2ktau}. Recall that $S \subseteq N_B(v)$, so every blue copy of $\wt{H_0}$ in $S$ yields a blue copy of $H_0$ containing $v$.
	Moreover, by \cref{lem:every-vertex-high-blue-deg}, we know that $v$ and every vertex of $S$ have at least $(1- \frac 1{k-1}-65h\sqrt \lambda)n \geq n/2$ blue neighbors. This shows that every labeled blue copy of $\wt{H_0}$ in $S$ extends to at least $(n/2)_{t-h}$ labeled blue copies of $H$ which contain $v$. Note that
	\[
		\frac{(n/2)_{t-h}}{(d)_{t-h}} = \frac{n/2}{d}\cdot \frac{n/2-1}{d-1} \dotsb \frac{n/2-t+h+1}{d-t+h+1} \geq \left(\frac{n/2}{d}\right)^{t-h} \geq \left(\frac 54\right)^{t-h} \geq \left(\frac 54\right)^{t/2} \geq t,
	\]
	where the second inequality uses that $d \leq \frac 25 n$ since $k \geq 4$, and the final inequality holds since $t \geq 100$.

	Combining the computations above, we find that the number of blue copies of $H$ containing $v$ is at least $t\tau(n)_{h-1}(d)_{t-h}$, which is at least $2t(k-1)^{1-t}(n)_{t-1}$ by the same computation as in \cref{eq:basic-d-lb}. This is a contradiction to \eqref{eq:local-turan-ub} for sufficiently large $n$.
\end{proof}
\cref{lem:high-red-deg} showed that each vertex $v$ has at least one part $V_i$ to which it has blue degree less than $\theta \ab{V_i}$. The next claim shows that in fact, there must be exactly one such part, in a strong sense: if $\ab{N_B(v) \cap V_i} < \theta \ab{V_i}$, then $v$ must have many blue neighbors in $V_j$ for all $j \neq i$.

\begin{claim}\label{lem:high-blue-deg}
	Let $v \in V(K_n)$, and let $U_i = N_B(v) \cap V_i$ denote the set of blue neighbors of $v$ in $V_i$, for $1 \leq i \leq k-1$. If $\ab{U_i} < \theta \ab{V_i}$ for some $i$, then $\ab{U_j} \geq (1-2k \theta) \ab{V_j}$ for all $j \neq i$.
\end{claim}
\begin{proof}
	By \cref{lem:every-vertex-high-blue-deg}, $v$ has at most $(\frac{1}{k-1}+65h\sqrt \lambda)n$ red neighbors. We also know that $v$ has at least $(1- \theta) \ab{V_i}-1$ red neighbors in $V_i$, and by (\ref{eq:part-size-bounds}),
	\[
		(1 -\theta)\ab{V_i} \geq (1- \theta)\left( \frac{1}{k-1} - \sqrt{2 \varepsilon} \right) n \geq \left( \frac{1}{k-1} - \sqrt{2 \varepsilon} - \theta \right) n.
	\]
	Hence, $v$ is incident to at most $(65h\sqrt\lambda+ \sqrt{2 \varepsilon}+\theta)n$ red edges with vertices in $V(K_n)\setminus V_i$. Since $65h\sqrt\lambda \leq \varepsilon$, $\sqrt \varepsilon \leq \theta/2$, and $\varepsilon \leq 1/3$, we have that $\lambda+\sqrt{2 \varepsilon}+\theta \leq 2 \theta$.  Since $\ab{V_j} \geq n/k$ for each $j \neq i$, we conclude that $v$ has at most $2k \theta\ab{V_j}$ red neighbors in $V_j$, as claimed.
\end{proof}

Using the previous two claims, we next show that every vertex has high red degree to its part in the partition, and high blue degree to all other parts. 
\begin{claim}\label{lem:vertices-ok}
	Let $v \in V_i$. Then $\ab{N_R(v) \cap V_i} \geq (1- \theta)\ab{V_i}$, and $\ab{N_B(v) \cap V_j} \geq (1- 2k \theta)\ab{V_j}$ for all $j \neq i$.
\end{claim}
\begin{proof}
	By \cref{lem:high-red-deg,lem:high-blue-deg}, we know that there exists an index $i' \in [k-1]$ such that $\ab{N_B(v) \cap V_{i'}} \leq \theta\ab{V_{i'}}$, and that $\ab{N_B(v) \cap V_j} \geq (1- 2k \theta)\ab{V_j}$ for all $j \neq i'$. We claim that $i=i'$. 

	We recall that since we assumed $V_1,\dots,V_{k-1}$ is a max $(k-1)$-cut of the blue graph, every vertex has at most as many blue neighbors in its part as in any other part. But (\ref{eq:part-size-bounds}) implies that $(1-2k \theta) \ab{V_j} > \theta \ab{V_{i'}}$ for all $j \neq i'$, since $3k \theta < \frac 12$ and $\sqrt \varepsilon <\frac{1}{4k}$. So $v$ has fewer blue neighbors in $V_{i'}$ than in any other part, implying that $i=i'$.
\end{proof}

We now know that the coloring is ``almost'' a Tur\'an coloring with respect to the partition $V_1 \sqcup \dotsb \sqcup V_{k-1}$: most internal edges are red and most other edges are blue. The next claim demonstrates that in fact, all internal edges are red.

Recall that $uu' \in E(H_0)$ is an edge whose deletion yields a $(k-1)$-colorable graph, and let $\wh{H_0}$ be this graph. Fix a proper $(k-1)$-coloring of $\wh{H_0}$, and note that in this proper coloring, $u$ and $u'$ must receive the same color, for otherwise we could extend it to a proper $(k-1)$-coloring of $H_0$. Suppose without loss of generality $u$ and $u'$ receive color $1$, and let the sizes of the color classes be $c_1+2,c_2,\dots,c_{k-1}$, so that $\sum c_i = h-2$.
\begin{claim}\label{lem:parts-red}
	For each $1 \leq i \leq k-1$, there is no blue edge inside $V_i$.
\end{claim}
\begin{proof}
	By relabeling the parts, it suffices to prove this for $i=1$. So suppose that $v,w$ are two vertices in $V_1$ such that the edge $vw$ is blue. For each $j > 1$, let $W_j$ denote the common blue neighborhood of $v$ and $w$ in $V_j$, and let $W_1 = V_1 \setminus \{v,w\}$. By \cref{lem:vertices-ok}, $v$ and $w$ are each incident to at most $2k \theta\ab{V_j}$ red edges with vertices in $V_j$, and hence $\ab{W_j} \geq (1- 4k \theta) \ab{V_j} \geq \frac 12 \ab{V_j}$ for all $j$. Additionally, by the same computation as in (\ref{eq:density-in-subsets}), we see that for every $1\leq i \neq j \leq k-1$,
	\[
		e_B(W_i, W_j) \geq \ab{W_i} \ab{W_j} - k^2 \varepsilon \ab{V_i} \ab{V_j} \geq \left( 1- 4k^2 \varepsilon \right) \ab{W_i} \ab{W_j}.
	\]
	Therefore, if we pick $c_j$ random vertices (with replacement) from $W_j$ for each $1 \leq j \leq k-1$, then by the union bound, they will form a blue homomorphic copy of $\wh{H_0} \setminus \{u,u'\}$ with probability at least $1-\binom{h-2}2 4k^2 \varepsilon\geq \frac 34$. Additionally, if $n$ is sufficiently large, then all these vertices will be distinct with probability at least $\frac 34$. Thus, the edge $vw$ lies in at least $\frac 12 \prod \ab{W_j}^{c_j}$ labeled blue copies of $\wh{H_0}$. Each vertex of every such copy of $\wh{H_0}$ has blue degree at least $(1- \frac 1{k-1}-65 h\sqrt \lambda)n$, so every such blue copy of $\wh{H_0}$ can be extended to a labeled blue copy of $H$ in at least $((1- \frac 1{k-1}-65 h\sqrt \lambda)n-h)_{t-h}$ ways. So in total, the number of labeled blue copies of $H$ containing the edge $vw$ is at least
	\[
		\left(\left(1- \frac 1{k-1}-65 h\sqrt \lambda\right)n-h\right)_{t-h} \cdot \frac 12 \prod_{j=2}^{k-1} \ab{W_j}^{c_j} \geq (2k)^{-h} \left(\frac 35\right)^t n^{t-2},
	\]
	since $(1- \frac 1{k-1}-65 h\sqrt \lambda)n-h > \frac 35 n$ and $\ab{W_j} \geq \frac 12 \ab{V_j} \geq \frac{1}{2k}n$.

	Now, suppose that we define a new coloring $\chi'$ by recoloring the edge $vw$ red. We claim that doing so decreases the total number of monochromatic copies of $H$, contradicting our defining assumption on $\chi$. To prove this, we need to upper-bound the number of labeled red copies of $H$ that are created when we recolor $vw$ red. The set of such $H$ consists of those in which $vw$ is a pendant edge and those in which $vw$ is one of the edges of $H_0$. There are at most $2 n^{h-1}(\frac 25 n)^{t-h-1}$ copies of the former type; indeed, we have two choices for which of $v$ and $w$ lies in $H_0$, at most $n^{h-1}$ choices for the remaining vertices of $H_0$, and at most $\ab{V_j}+2 k \theta (n-\ab{V_j}) \leq \frac 25 n$ choices for each other pendant vertex, since \cref{lem:high-blue-deg} implies that every vertex in part $V_j$ has at most $2k \theta (n-\ab{V_j})$ red neighbors outside $V_j$. By a similar argument, there are at most $h! n^{h-2}(\frac 25 n)^{t-h}$ copies of the latter type; there are at most $n^{h-2}$ choices for the other vertices of $H_0$, at most $h!$ automorphisms of $H_0$, and at most $\frac 25 n$ choices for each pendant vertex.

	Therefore, we see that recoloring $vw$ red produces at most
	\begin{align*}
		2n^{h-1}\left( \frac 25 n\right)^{t-h-1} + h! n^{h-2}\left( \frac 25 n\right)^{t-h} < (2h)^h\left(\frac 25\right)^t n^{t-2}
	\end{align*}
	new labeled red copies of $H$. We note that since $\frac32 \geq \sqrt 2$ and $2^x > x$ for all $x \geq 0$,
	\[
		\left(\frac 32\right)^t \geq 2^{t/2} \geq 2^{4kh^2} > (4kh)^h.
	\]
	This is equivalent to $(2k)^{-h}(\frac 35)^t n^{t-2} > (2h)^h (\frac 25)^t n^{t-2}$.
	Therefore, $\chi'$ has strictly fewer monochromatic copies of $H$ than $\chi$, a contradiction.
\end{proof}

The previous claim showed that each part $V_i$ contains only red edges. In particular, we see that the blue graph is $(k-1)$-colorable, and thus there are no blue copies of $H$. The next claim shows that moreover, the edges between $V_i$ and $V_j$ must all be blue.
\begin{claim}\label{lem:across-blue}
	For every $1 \leq i \neq j \leq k-1$, all edges between $V_i$ and $V_j$ are blue.
\end{claim}
\begin{proof}
	Suppose for contradiction that there is a red edge $vw$, where $v \in V_i$ and $w \in V_j$. Since every edge in $V_i$ is red and $\ab{V_i} > t$ for $n$ sufficiently large, we see that $vw$ must lie in at least one red copy of $H$. However, if we recolor $vw$ blue, then it will not lie in any blue copy of $H$, since recoloring it will still maintain the property that the blue graph is $(k-1)$-colorable. This shows that recoloring $vw$ blue must strictly decrease the number of monochromatic copies of $H$, contradicting our choice of $\chi$.
\end{proof}

At this point, we've found that if the coloring $\chi$ minimizes the number of copies of $H$, then its red graph consists of a disjoint union of $k-1$ red cliques, and its blue graph is complete $(k-1)$-partite. This shows that there are no blue copies of $H$, and the number of labeled red copies of $H$ equals $\sum_{i=1}^{k-1} (\ab{V_i})_t$, since each $V_i$ spans a red clique. The function
\[
	f(x)=
	\begin{cases}
		(x)_t &x \geq t\\
		0&x<t
	\end{cases}
\]
is convex, and agrees with the function $(x)_t$ whenever $x$ is a non-negative integer. Thus, by Jensen's inequality, the quantity $\sum_{i=1}^{k-1} (\ab{V_i})_t = \sum_{i=1}^{k-1}f(\ab{V_i})$ is minimized when all the quantities $\ab{V_i}$ are as equal as possible. In other words, the unique coloring $\chi$ which minimizes the number of monochromatic copies of $H$ is the Tur\'an coloring, which completes the proof of \cref{thm:main}.
\end{proof}

\section{Three-color bonbons}\label{sec:three-colors}

In this section, we prove \cref{thm:3-color}.
Before proceeding with the proof, we briefly discuss how the three-color case differs from the two-color case, and where new ideas are needed in the proof. The first difference is that in the two-color case, we had three sets $V_R, V_B, V_{RB}$, and it was fairly easy to show that two of them must be very small (\cref{lem:V0-small,lem:V_R-or-V_B-small}). In the three-color setting, we start with seven sets, corresponding to the non-empty subsets of $\{R,B,Y\}$. As before, it is straightforward to show that four of these must be small, namely $V_R,V_B,V_Y$, and $V_{RBY}$. However, showing that two of the remaining sets are also small requires a new idea, which is where the assumption that $k$ is polite arises: we show that if two of the remaining sets are both large, then we can find many monochromatic $K_k$ by ``gluing together'' monochromatic $K_{k/2}$ in the two large sets. This step is done in \cref{claim:VRB-or-VRY-small}.

The other big difference is that in the two-color case,
we could pick the ``error parameter'' $\lambda$ to be very small with respect to $1/k$, and could thus prove \cref{lem:every-vertex-high-blue-deg} directly after \cref{claim:V_B-few-high-deg}. In the three-color setting, there are two such ``error parameters'': $\lambda$, which controls the size of $V_{RBY}$, and $\eta$, which controls the sizes of the remaining small sets. We can again ensure that $\lambda$ is very small, but $\eta$ is by necessity reasonably large---much smaller than $1/k$, but larger than $1/r(k,k)$. This means that the sets whose size is controlled by $\eta$ are not small enough for the argument of \cref{lem:every-vertex-high-blue-deg} to work. In order to get around this, we first prove that these sets are actually \emph{empty}, at which point the argument of \cref{lem:every-vertex-high-blue-deg} can go through.

We now proceed with the proof of \cref{thm:3-color}.
So we fix a sufficiently large polite integer $k$, a sufficiently large integer $t$, and let $H$ be obtained from $K_k$ by appending $t-k$ pendant edges. As the overall structure of the proof is broadly similar to that of \cref{thm:main}, but with some added difficulties, we omit or shorten the proofs that are essentially identical to those presented in \cref{sec:proof-main}. Along the same lines, we will keep less careful track of the parameters, only enforcing the hierarchy
\[
	\eta \ll \frac 1k, \qquad \gamma \ll \frac 1k, \qquad \theta \ll \frac 1k, \qquad \varepsilon \ll \theta, \qquad \lambda \ll \eta,   \qquad \lambda \ll \varepsilon,
\]
as well as choosing $\tau = (1+\lambda)^{-t}$, and then ensuring that $t$ is sufficiently large so that $\tau \ll \lambda$.

Suppose we have a red/blue/yellow-coloring $\chi$ of $E(K_n)$ with the minimum number of monochromatic $H$. We wish to prove that the coloring is Ramsey-blowup-like, meaning that one of the colors spans $r(k,k)-1$ disjoint cliques whose sizes differ by at most one, and the remaining two colors contain no monochromatic $K_k$. Let $r = r(k,k)$, let $\lambda = \lambda(k)>0$, and let $d = (1+\lambda)n/(r-1)$. We recall that by the Ramsey-blowup coloring, we have
\begin{equation}\label{eq:ramsey-blowup-ub}
	m_3(H,n) \leq (r-1)^{1-t} (n)_t.
\end{equation}

We now record without proof the three-color analogues of \cref{lem:degree-regular,lem:clique-multiplicity,lem:few-H0-in-high-deg}. The proofs are identical to those given in the two-color case; the main significant observation is that any upper bound on $c_3(H)$, such as the one given by the Ramsey-blowup coloring in (\ref{eq:ramsey-blowup-ub}), yields a version of \cref{lem:few-H0-in-high-deg}, which is why the parameter $d$ defined above appears in \cref{lem:3-color-few-H0}.

\begin{lem}[Three-color analogue of \cref{lem:degree-regular}]\label{lem:3-color-degree-regular}
	For every graph $H$ on $t$ vertices, there exists a constant $C>0$ such that the following holds. For any three-coloring $\chi$ of $E(K_n)$ with the minimum number of monochromatic copies of $H$, and for any $v \in V(K_n)$,
	\[
		\left(1- \frac Cn\right)\frac tn m_3(H,n) \leq m_v(H,\chi) \leq \left(1+\frac Cn\right)\frac tn m_3(H,n),
	\]
	where $m_v(H,\chi)$ denotes the number of monochromatic labeled copies of $H$ containing $v$.
\end{lem}
\begin{lem}[Three-color analogue of \cref{lem:clique-multiplicity}]\label{lem:3-color-clique-multiplicity}
	Let $H_1,H_2,H_3$ be graphs on $h_1,h_2,h_3$ vertices, respectively, and suppose that $h_1,h_2,h_3\leq h$. For any $n \geq 27^{h^2}$, any red/blue/yellow-coloring of $E(K_n)$ contains at least $27^{-h^2}(n)_{h_1}$ labeled red copies of $H_1$, or at least $27^{-h^2}(n)_{h_2}$ labeled blue copies of $H_2$, or at least $27^{-h^2}(n)_{h_3}$ labeled yellow copies of $H_3$.
\end{lem}
\begin{rem}
	The constant $27$ appears in this statement because the best known upper bound on the three-color Ramsey number of $K_k$ is $(27-o(1))^k$. In general, the $q$-color version of such a statement would involve the constant $q^q$.
\end{rem}

\begin{lem}[Three-color analogue of \cref{lem:few-H0-in-high-deg}]\label{lem:3-color-few-H0}
Suppose $n$ is sufficiently large. Let $S\subseteq V(K_n)$ be a set of vertices, each with at least $d$ blue neighbors. Then $S$ contains fewer than $\tau(n)_k$ labeled blue copies of $K_k$.
\end{lem}

In addition to these three basic lemmas, we also need the following result, which is proved by combining \cref{lem:supersaturation} and Ramsey's theorem.
\begin{lem}\label{lem:3-color-supersaturation}
    For all integers $k,\ell$, there exists $\gamma = \gamma(k,\ell)>0$ such that the following holds for all sufficiently large $n$ and any red/blue/yellow-coloring of $E(K_n)$. If there are fewer than $\frac{1}{r(k,\ell)}\binom n2$ yellow edges, then there are at least $\gamma(n)_k$ labeled blue copies of $K_k$, or at least $\gamma(n)_\ell$ labeled red copies of $K_\ell$.
\end{lem}
\begin{proof}
    Let $G$ be the $n$-vertex graph whose edges are the red and blue edges of $K_n$. By assumption, the number of edges in $G$ is at least
	\[
		\left(1 - \frac{1}{r(k,\ell)}\right)\binom n2 \geq \left(1- \frac{1}{r(k,\ell)-1} + \frac{1}{r(k,\ell)^2}\right) \binom n2.
	\]
	Therefore, by \cref{lem:supersaturation}, $G$ contains at least $\gamma_0(n)_{r(k,\ell)}$ labeled copies of $K_{r(k,\ell)}$, for some $\gamma_0>0$ depending only on $k,\ell$. In the original coloring of $E(K_n)$, each such copy of $K_{r(k,\ell)}$ in $G$ contains at least one blue copy of $K_k$ or at least one red copy of $K_\ell$. Suppose first that least half of them contribute a blue copy of $K_k$. Each labeled blue copy of $K_k$ extends to a labeled copy of $K_{r(k,\ell)}$ in $G$ in at most $\frac{r(k,\ell)!}{k!}(n-k)_{r(k,\ell)-k}$ ways. So in this case, the original coloring of $K_n$ contains at least $\frac{\gamma_0}{2 r(k,\ell)!}(n)_k$ labeled blue copies of $K_k$. Similarly, if at least half the copies of $K_{r(k,\ell)}$ in $G$ contribute a red $K_\ell$, we find at least $\frac{\gamma_0}{2 r(k,\ell)!}(n)_\ell$ labeled red copies of $K_\ell$. In either case, we get the desired result, where $\gamma = \frac{\gamma_0}{2 r(k,\ell)!}>0$ depends only on $k$ and $\ell$.
\end{proof}

We also record here a result of Xu, Shao, and Radziszowski \cite{MR2801235} on the difference between consecutive Ramsey numbers. This inequality improves by an additive constant of $1$ a classical result of Burr, Erd\H os, Faudree, and Schelp \cite{MR992396}.
\begin{lem}[\cite{MR2801235}]\label{lem:consecutive-ramsey-difference}
	For any $k \geq 5$, we have
	\[
		r(k,k) \geq r(k,k-1) + 2k-2.
	\]
\end{lem}
We remark that we will only use the (much) weaker result that $r(k,k) \geq r(k,k-1)+6$ for sufficiently large $k$. However, we state \cref{lem:consecutive-ramsey-difference} both because it gives the best known lower bound on $r(k,k) - r(k,k-1)$, and because the weakness of this bound shows how far we are from proving \cref{conj:assumptions}, namely that all sufficiently large integers are polite.

We are now ready to begin the proof of \cref{thm:3-color}.

\begin{proof}[Proof of \cref{thm:3-color}]
We fix a coloring $\chi$ of $E(K_n)$ with colors red, blue, and yellow, and assume that $\chi$ has the minimum number of monochromatic copies of $H$ among all such colorings.
By \cref{lem:3-color-degree-regular} and \eqref{eq:ramsey-blowup-ub}, we have that for any vertex $v \in V(K_n)$,
\begin{equation}\label{eq:local-ramsey-blowup-ub}
	m_v(H,\chi) \leq \left(1+\frac Cn\right) \frac tn m_3(H,n) \leq \left(1+\frac Cn\right)t (r-1)^{1-t} (n-1)_{t-1},
\end{equation}
where $C$ is a constant depending only on $H$.

Let
\[
	\eta= \frac{8r(k,k/2)}{r-1},
\]
and observe that for sufficiently large $k$,
\begin{equation}\label{eq:eta-lb}
	\eta\geq \frac{3k}{r-1}
\end{equation}
since $8r(k,k/2) \geq 3 k$ for all large $k$, as $r(k,k/2)$ grows exponentially. Note too that if $k$ is polite, then
\begin{equation}\label{eq:polite-eta-bound}
	\eta \leq 2^{-28}
\end{equation}
and
\begin{equation}\label{eq:assumptions-consequence}
	\frac{r-1}{r(k,k-1)} \geq 1+ 25 \left(\frac{r(k,{k/2})}{r}\right)^{1/4} > 1+12 \eta^{1/4}.
\end{equation}

For every subset $S$ of $\{R,B,Y\}$, let $V_S$ denote those vertices with degree at least $d$ in each color in $S$, but no other colors. Note that $V_\varnothing$ is empty, since $3d<n-1$ and thus every vertex has at least $d$ neighbors in at least one color. Our next three claims, which are three-color analogues of \cref{lem:V0-small,lem:V_R-or-V_B-small}, show that six of the remaining seven sets $V_S$ are small.

\begin{claim}\label{claim:VRBY-small}
	If $t$ is sufficiently large in terms of $k$, then $\ab{V_{RBY}}\leq \lambda n$.
\end{claim}
\begin{proof}
	This is proved in the same way as \cref{lem:V0-small}.
\end{proof}

\begin{claim}
	$\ab{V_R \cup V_B \cup V_Y} \leq 3 \eta n$.
\end{claim}
\begin{proof}
	Suppose for contradiction that this is false, and assume without loss of generality that $\ab{V_R}> \eta n$. Every vertex in $V_R$ has blue and yellow degrees less than $d$, and therefore has red degree at least $n-2d = (1- 2(1+\lambda)/(r-1))n > (1-3/r)n$. In particular, every vertex in $V_R$ has at least
	\[
		\ab{V_R}-\frac{3}r n > \left(1 - \frac{3}{\eta r}\right)\ab{V_R} \geq \left(1- \frac{1}{k}\right)\ab{V_R} \geq \left(1- \frac{1}{k-1}+\frac{1}{k^2}\right)\ab{V_R}
	\]
	red neighbors in $V_R$, where the second inequality follows from (\ref{eq:eta-lb}). Therefore, by \cref{lem:supersaturation}, we see that $V_R$ contains at least $\gamma (\ab{V_R})_k$ red labeled copies of $K_k$, where $\gamma$ depends only on $k$. Any such copy of $K_k$ extends to at least $(n-2d-k)_{t-k}$ red copies of $H$, since each vertex in $V_R$ has red degree at least $n-2d$, and we subtract $k$ to ensure we don't pick one of the $k$ vertices from the fixed copy of $K_k$. So in total, the number of red labeled copies of $H$ is at least
	\[
		\gamma (\ab{V_R})_k (n-2d-k)_{t-k} \geq (1-o(1)) \gamma \eta^k \left(1- \frac 3r\right)^{t-k} (n)_t \geq \gamma \eta^k 2^{k-t} (n)_t
	\]
	for $n$ sufficiently large.
	Since we chose $t$ sufficiently large relative to $k$ (and thus sufficiently large relative to $\gamma$ and $\eta$, which depend only on $k$), we have that
	\[
		\gamma \eta^k 2^{k-t} > 2t(r-1)^{1-t},
	\]
	as the left-hand side decays exponentially in $t$ with base $2$, whereas the right-hand side decays exponentially with base $r-1>2$. Thus, the number of red labeled copies of $H$ we've found is strictly more than the upper bound given in (\ref{eq:ramsey-blowup-ub}), which is the desired contradiction.
\end{proof}

At this point we've found that $V_\varnothing \cup V_R \cup V_B \cup V_Y \cup V_{RBY}$ is small, so almost all vertices lie in $V_{RB} \cup V_{RY} \cup V_{BY}$. Our next lemma shows that in fact, two of these three sets must also be small. 
\begin{claim}\label{claim:VRB-or-VRY-small}
	$\min \{\ab{V_{RB}}, \ab{V_{RY}}\} < \eta n$.
\end{claim}
\begin{proof}
	Suppose for contradiction that $\ab{V_{RB}}, \ab{V_{RY}} \geq \eta n$. Note that these two sets are disjoint by definition. Let $c = \frac{1}{8k}$, and note that $\frac1{4r(k,k/2)}\leq c$ for sufficiently large $k$, since $r(k,k/2)$ grows exponentially in $k$. Since every vertex in $V_{RB}$ has at most $d$ yellow neighbors, the total number of yellow edges between $V_{RB}$ and $V_{RY}$ is at most
	\[
		d \ab{V_{RB}} \leq \frac{d}{\eta n} \ab{V_{RB}} \ab{V_{RY}} \leq \frac{2n/(r-1)}{\eta n} \ab{V_{RB}} \ab{V_{RY}} = \frac{1}{4 r(k,k/2)} \ab{V_{RB}} \ab{V_{RY}} \leq c \ab{V_{RB}} \ab{V_{RY}}.
	\]
	For the same reason, the number of blue edges between $V_{RB}$ and $V_{RY}$ is at most $d\ab{V_{RY}} \leq c \ab{V_{RB}} \ab{V_{RY}}$.
	Therefore, at least a $1-2c$ fraction of the edges between $V_{RB}$ and $V_{RY}$ are red. 

	At this point, our goal is to build many red copies of $K_k$ in $V_{RB} \cup V_{RY}$, which will yield a contradiction to \cref{lem:3-color-few-H0}. To do so, we first build many red copies of $K_{k/2}$ in $V_{RB}$. Since most of the edges between $V_{RB}$ and $V_{RY}$ are red, each such red $K_{k/2}$ can be glued to many red copies of $K_{k/2}$ in $V_{RY}$, to form the large collection of red $K_k$ in $V_{RB} \cup V_{RY}$. We now proceed with the technical details of this argument.

	Let $V_{RB}' \subseteq V_{RB}$ consist of those vertices with at least $1-4c \ab{V_{RY}}$ red neighbors in $V_{RY}$, so that $\ab{V_{RB}'}\geq \frac 12 \ab{V_{RB}}$. Recall that every vertex in $V_{RB}'$ has fewer than $d$ yellow neighbors, and thus fewer than $\frac{2d}{\eta n}\ab{V_{RB}'}$ yellow neighbors in $V_{RB}'$. We also have that
	\[
		\frac{2d}{\eta n}< \frac{4}{\eta r} = \frac{1}{2r(k,k/2)}
	\]
	by our choice of $\eta$. 
	We now apply \cref{lem:3-color-supersaturation} to the induced coloring on $V_{RB}'$ with $\ell=k/2$, which we may do since the above implies that the yellow edge density is less than $\frac 1{r(k,k/2)}$. We conclude that $V_{RB}'$ contains at least $\gamma(\ab{V_{RB}'})_k$ labeled blue $K_k$ or at least $\gamma(\ab{V_{RB}'})_{k/2}$ labeled red $K_{k/2}$, for some $\gamma>0$ depending only on $k$.
	On the other hand, by \cref{lem:3-color-few-H0}, we see that $\ab{V_{RB}'}$ must contain fewer than $\tau (n)_k$ blue $K_k$. Recalling that we chose $t$ sufficiently large, and thus $\tau$ sufficiently small, we see that $\tau<\gamma$, implying that the former case cannot occur. So $V_{RB}'$ contains at least $\gamma (\ab{V_{RB}}')_{k/2}$ labeled red copies of $K_{k/2}$.

	Fix such a red $K_{k/2}$, and call it $Q$. As every vertex of $Q$ has at least $(1-4c)\ab{V_{RY}}$ red neighbors in $V_{RY}$, the common red neighborhood of $Q$ in $V_{RY}$ has size at least $(1-4kc)\ab{V_{RY}} = \frac 12 \ab{V_{RY}}$. Let $V_{RY}' \subseteq V_{RY}$ be this common red neighborhood of $Q$. By exactly the same argument as above (but with the roles of blue and yellow swapped), we see that $V_{RY}'$ must contain at least $\gamma (\ab{V_{RY}'})_{k/2}$ red $K_{k/2}$. Since this works for every choice of $Q$, we conclude that $V_{RB} \cup V_{RY}$ contains at least $\gamma^2 (\ab{V_{RB}'})_{k/2} (\ab{V_{RY}'})_{k/2} \geq \gamma' (n)_k$ red $K_k$, for some $\gamma'>0$ depending only on $k$. As we chose $t$ sufficiently large, we have that $\tau<\gamma'$. This yields a contradiction to \cref{lem:3-color-few-H0}, since every vertex in $V_{RB} \cup V_{RY}$ has at least $d$ red neighbors.
\end{proof}

By permuting the colors, we may assume without loss of generality that $\ab{V_{RY}}, \ab{V_{BY}} \leq \eta n$. Putting together the previous three claims, we conclude that $\ab{V_{RB}} \geq (1- 6 \eta)n$. 

Before continuing to the next claim, we record some simple bounds for future convenience. First, we note that for $0 \leq x \leq \frac{1}{100}$, we have that
\[
	\frac{1}{(1+\sqrt x)(1-6 x)} < 1- x.
\]
We know that $\eta \leq \frac{1}{100}$ by (\ref{eq:polite-eta-bound}), so
\[
	\frac{1}{(1+\sqrt \eta)(1-6 \eta)} < 1- \eta = 1- \frac{8r(k,k/2)}{r-1} < 1 - \frac 1r = \frac{r-1}{r}.
\]
By rearranging, and by picking $\lambda$ sufficiently small with respect to $\eta$, we find that
\[
	\frac{1+\lambda}{(1-6 \eta)(r-1)} < \frac{1+\sqrt \eta }{r}
\]
which implies that
\begin{equation}\label{eq:d-VRB-bound}
	d = \frac{1+\lambda}{r-1}n < \frac{1+\sqrt \eta }{r} \ab{V_{RB}}
\end{equation}
since $\ab{V_{RB}} \geq (1-6 \eta)n$. 
Our next claim is the three-color analogue of \cref{claim:V_B-few-high-deg}. It says few vertices in $V_{RB}$ have low yellow degree.
\begin{claim}\label{claim:very-low-yellow-deg}
	Let $V_{RB}' \subseteq V_{RB}$ denote the set of vertices in $V_{RB}$ with yellow degree less than $\frac{1- \eta^{1/4}}{r} \ab{V_{RB}}$. Then $\ab{V_{RB}'} \leq \eta^{1/4} n$.
\end{claim}
\begin{proof}
	The proof is very similar to that of \cref{claim:V_B-few-high-deg}. Suppose first that there are fewer than $\frac 1r \binom{\ab{V_{RB}}}2$ yellow edges in $V_{RB}$. Then by applying \cref{lem:3-color-supersaturation} with $k=\ell$ to the induced coloring on $V_{RB}$, we find that $V_{RB}$ must contain at least $\gamma (\ab{V_{RB}})_k$ labeled red or blue copies of $K_k$, for some $\gamma>0$ depending only on $k$.
	But this is a contradiction to \cref{lem:3-color-few-H0} because every vertex in $V_{RB}$ has red and blue degree at least $d$, and because we chose $t$ sufficiently large, and thus $\tau<\gamma$. 

	Therefore, the number of yellow edges in $V_{RB}$ is at least $\frac{1}{r}\binom{\ab{V_{RB}}}2 = \frac{\ab{V_{RB}}}{2r}(\ab{V_{RB}}-1)$. Every vertex in $V_{RB}$ has at most $d$ yellow neighbors in $V_{RB}$, and every vertex in $V_{RB}'$ has fewer than $\frac{1- \eta^{1/4}}{r}\ab{V_{RB}}$ yellow neighbors in $V_{RB}$. Summing this up over all vertices in $V_{RB}$, we see that the total number of yellow edges in $V_{RB}$ is at most
	\[
		\frac12 \left(\sum_{v \in V_{RB}'} \frac{1-\eta^{1/4}}{r}\ab{V_{RB}} +\hspace{-10pt} \sum_{v \in V_{RB} \setminus V_{RB}'} d\right) < \frac {\ab{V_{RB}}}{2r} \left(({1-\eta^{1/4}})\ab{V_{RB}'} + ({1+\sqrt \eta}) (\ab{V_{RB}}-\ab{V_{RB}'})\right),
	\]
	where we use (\ref{eq:d-VRB-bound}).
	Comparing our lower and upper bounds for the number of yellow edges in $V_{RB}$, we find that
	\[
		\ab{V_{RB}}-1 < (1+\sqrt \eta) \ab{V_{RB}} - (\eta^{1/4}+\sqrt \eta)\ab{V_{RB}'} \leq (1+\sqrt \eta) \ab{V_{RB}} - \eta^{1/4} \ab{V_{RB}'},
	\]
	which implies that
	\[
		\ab{V_{RB}'} \leq \frac{\sqrt \eta}{\eta^{1/4}}\ab{V_{RB}} \leq {\eta^{1/4}} n.\qedhere
	\]
\end{proof}

Our next claim does not have a direct analogue in the two-color setting, and this is because of an important difference between the two cases discussed at the beginning of \cref{sec:three-colors}. Namely, in the case of two colors, we picked $\lambda$ to be very small with respect to $\frac 1k$, and could thus prove \cref{lem:every-vertex-high-blue-deg} directly after \cref{claim:V_B-few-high-deg} (whose analogue we have just proven). In the three-color setting, $\lambda$ is still very small, but $\eta$ is actually fairly large compared to $\frac 1r$. So if we try to mimic the proof of \cref{lem:every-vertex-high-blue-deg} using the information we have so far, we will not be able to deduce a contradiction, because the vertices in $V_R \cup V_B \cup V_Y \cup V_{RY} \cup V_{BY}$ actually may have large yellow degree.

We do eventually prove a three-color analogue of \cref{lem:every-vertex-high-blue-deg}, in \cref{claim:no-vtx-high-yellow-deg}. But in order for the proof to work, we first show that all the sets $V_R, V_B, V_Y, V_{RY}, V_{BY}$ are empty, which is the content of the next three claims. The first says that no vertex has many yellow neighbors in $V_{RB}$. 
Recall that $s_1 \geq \dotsb \geq s_k$ denote the number of pendant edges on the vertices of $K_k$ in $H$.

\begin{claim}\label{claim:few-yellow-nbrs-in-VRB}
	Every vertex has at most $2\eta^{1/4} n$ yellow neighbors in $V_{RB}$. 
\end{claim}
\begin{proof}
	Suppose for contradiction that some $v \in V(K_n)$ has at least $2\eta^{1/4}n$ yellow neighbors in $V_{RB}$, and let $S = N_Y(v) \cap V_{RB}$ be the set of these yellow neighbors. Let $T = S \setminus V_{RB}'$, where $V_{RB}'$ is the set of vertices with fewer than $\eta^{1/4} n$ yellow neighbors, as in \cref{claim:very-low-yellow-deg}. By \cref{claim:very-low-yellow-deg}, we see that $\ab T \geq 2 \eta^{1/4}n - \eta^{1/4} n = \eta^{1/4}n$. Since $T \subseteq V_{RB}$, \cref{lem:3-color-few-H0} implies that $T$ contains fewer than $\tau (n)_k$ red $K_k$ and fewer than $\tau (n)_k$ blue $K_k$. Since we chose $t$ sufficiently large so that $\tau (n)_k \leq 27^{-k^2}(\ab T)_k$, we conclude from \cref{lem:3-color-clique-multiplicity} that $T$ contains at least $27^{-k^2}(\ab T)_{k-1}$ labeled yellow copies of $K_{k-1}$.  Note that $(\ab T)_{k-1} \geq (1-o(1)) \eta^{(k-1)/4} (n)_{k-1} \geq \eta^{k/4}(n)_{k-1}$ for $n$ sufficiently large. Thus, for sufficiently large $n$, we find that $T$ contains at least $27^{-k^2} \eta^{k/4} (n)_{k-1}$ labeled yellow copies of $K_{k-1}$. Fix any such copy of $K_{k-1}$ in $T$. It can be extended to a labeled yellow copy of $H$ in at least
	\[
		(2 \eta^{1/4}n - k)_{s_1} \left(\frac{1-\eta^{1/4}}{r}n-k\right)_{s_2+\dotsb+s_k}
	\]
	ways, since $v$ has at least $2 \eta^{1/4}n - k$ yellow neighbors besides the $k$ vertices already chosen, and every other vertex of the $K_k$ lies in $V_{RB} \setminus V_{RB}'$, and hence has at least $\frac{1-\eta^{1/4}}{r}n$ yellow neighbors. Note that $2 \eta^{1/4} > \frac 1r + \eta^{1/4}$, and so the expression above is at least
	\[
		(1-o(1))r^{k-t} (r \eta^{1/4})^{s_1} (1- \eta^{1/4})^{s_2+\dotsb+s_k} (n)_{t-k}.
	\]
	Since $r \geq 4$, we have that $r \eta^{1/4} \geq r (r-1)^{-1/4} \geq \sqrt r \geq 2^{k/4}$, using the bound $r \geq 2^{k/2}$ due to Erd\H os \cite{MR19911}. Additionally, by (\ref{eq:polite-eta-bound}), we have that $\eta \leq 2^{-20}= 32^{-4}$, so that $1- \eta^{1/4} \geq 1- \frac 1{32} \geq 2^{-1/16}$. Putting this all together, we see that
	\[
		(r \eta^{1/4})^{s_1} (1- \eta^{1/4})^{s_2+\dotsb+s_k} \geq 2^{\frac{ks_1}{4}} \cdot 2^{- \frac{s_2+\dotsb+s_k}{16}} = 2^{\frac{1}{16}(4ks_1 - s_2 - \dotsb s_k)} \geq 2^{t/32},
	\]
	using the fact that $ks_1 \geq s_1 + \dotsb + s_k = t-k$, so that $4ks_1 - (s_2 + \dotsb + s_k) \geq 3ks_1 \geq t$.

	Now, we recall that $T$ contains at least $27^{-k^2} \eta^{k/4} (n)_{k-1}$ labeled yellow copies of $K_{k-1}$, so in total, the number of yellow labeled copies of $H$ containing $v$ is at least
	\[
		(1-o(1)) 27^{-k^2} \eta^{k/4} r^{k-t} 2^{t/32} (n)_{t-1},
	\]
	which contradicts  the bound (\ref{eq:local-ramsey-blowup-ub}) for large $n$, since we chose $t$ sufficiently large with respect to $k$.
\end{proof}

As a simple corollary of the previous claim, we can conclude that $V_Y$ is empty.
\begin{claim}\label{claim:VY-empty}
	$V_Y= \varnothing$.
\end{claim}
\begin{proof}
	Suppose that there is some vertex $v \in V_Y$. Since $v$ has fewer than $d$ red and fewer than $d$ blue neighbors, and fewer than $2 \eta^{1/4}n$ yellow neighbors in $V_{RB}$ by \cref{claim:few-yellow-nbrs-in-VRB}, we have that
	\[
		2d + 2\eta^{1/4}n > \ab{V_{RB}} \geq (1-6 \eta)n \geq \frac n2 \geq 2d+ 2 \eta^{1/4} n
	\]
	a contradiction, where we have that $\frac{1+\lambda}{r-1}+ \eta^{1/4} < \frac 14$ for $k$ sufficiently large, because of (\ref{eq:polite-eta-bound}).
\end{proof}

Similarly, our next claim shows that every remaining set $V_S$ is empty, with the exception of $V_{RB}$ and $V_{RBY}$.

\begin{claim}\label{claim:no-r-or-b-sets}
	$V_R = V_B = V_{RY} = V_{BY} =  \varnothing$.
\end{claim}
\begin{proof}
	We prove that $V_R \cup V_{RY} =\varnothing$; the exact same proof, by reversing the roles of red and blue, also shows that $V_B \cup V_{BY} =\varnothing$. Suppose that $v \in V_R \cup V_{RY}$. We know that $v$ has at most $d$ blue neighbors, and by \cref{claim:few-yellow-nbrs-in-VRB}, $v$ also has at most $2 \eta^{1/4}n$ yellow neighbors in $V_{RB}$. So
	\[
		\ab{N_R(v)\cap V_{RB}} \geq \ab{V_{RB}} - d - 2 \eta^{1/4}n \geq n  - 6 \eta n -3 \eta^{1/4}n \geq \left(1- 10 \eta^{1/4}\right)n.
	\]
	Let $T=N_R(v) \cap V_{RB}$. Every vertex in $T \subseteq V_{RB}$ has yellow degree at most
	\[
		d= \frac{d}{\ab T} \ab T\leq \frac{1+\lambda}{(r-1)(1-10\eta^{1/4})} \ab T \leq  \frac{(1+\lambda)(1+11 \eta^{1/4})}{r-1}\ab T\leq \frac{1+12 \eta^{1/4}}{r-1}\ab T.
	\]
	Here, the second inequality follows since $\frac{1}{1-10x} \leq 1+11x$ for all $0 \leq x \leq 2^{-7}$ and since $\eta^{1/4} \leq 2^{-7}$ by (\ref{eq:polite-eta-bound}), and the final inequality holds since we chose $\lambda$ sufficiently small with respect to $\eta$.
	Note that (\ref{eq:assumptions-consequence}) is equivalent to
	\[
		\frac{1+12 \eta^{1/4}}{r-1} < \frac{1}{r(k,k-1)}.
	\]
	Therefore, by \cref{lem:3-color-supersaturation} applied with $\ell=k-1$, we find that $T$ contains at least $\gamma(\ab T)_k$ labeled blue copies of $K_k$ or at least $\gamma(\ab T)_{k-1}$ labeled red copies of $K_{k-1}$, for some $\gamma>0$ depending only on $k$.
	The former outcome is impossible by \cref{lem:3-color-few-H0} since every vertex in $T$ has blue degree at least $d$. The latter outcome is also impossible, by the same argument as in \cref{lem:high-red-deg} (which is essentially the same as the argument in \cref{claim:few-yellow-nbrs-in-VRB} or \cref{lem:every-vertex-high-blue-deg}). Namely, any red copy of $K_{k-1}$ in $T$ extends to at least $(d-t)_{t-k-1}$ red copies of $H$ containing $v$, since both $v$ and every vertex in $T$ have red degree at least $d$. So if there are $\gamma (\ab T)_{k-1}$ labeled red $K_{k-1}$ in $T$, then we have at least $\gamma(\ab T)_{k-1} (d-t)_{t-k-1}$ labeled red copies of $H$ containing $v$, contradicting the bound (\ref{eq:local-ramsey-blowup-ub}).
\end{proof}

The following is the analogue of \cref{lem:every-vertex-high-blue-deg}, and is proved in the same way. 
\begin{claim}\label{claim:no-vtx-high-yellow-deg}
	Every vertex has yellow degree less than $(\frac{1}{r-1}+65k\sqrt \lambda)n$. 
\end{claim}
In addition to knowing that no vertex has yellow degree much greater than $d$, we will also need to know that every vertex has red and blue degree noticeably above $d$, which is the content of the next claim.
\begin{claim}\label{claim:no-very-low-r-b-deg}
	Every vertex has red and blue degree at least $\frac{2k-6}{r-1}n$.
\end{claim}
\begin{proof}
	Suppose for contradiction that some vertex $v$ has red degree less than $\frac{2k-6}{r-1}n$ (the case of a vertex of low blue degree follows by interchanging the roles of red and blue). By \cref{claim:no-vtx-high-yellow-deg}, $v$ has fewer than $(\frac{1}{r-1} + 65k \sqrt \lambda)n < \frac{2}{r-1}n$ yellow neighbors, since we chose $\lambda$ small with respect to $k$, and thus with respect to $r$. So we conclude that $v$ has at least $(1- \frac{2k-4}{r-1})n = \frac{r-2k+3}{r-1}n$ blue neighbors. Let $T$ be the blue neighborhood of $v$.
	Every vertex in $T$ has yellow degree less than
	\[
		\left(\frac{1}{r-1} + 65k \sqrt \lambda\right) n \leq \left(\frac{1}{r-1}+65k \sqrt \lambda\right) \frac{r-1}{r-2k+3}\ab T < \left(\frac{1}{r-2k+3} + 130 k \sqrt \lambda\right) \ab T,
	\]
	since $r-1 < 2(r-2k+3)$ for sufficiently large $k$, since $r$ grows exponentially in $k$. By \cref{lem:consecutive-ramsey-difference}, we know that $r-2k+3\geq r(k,k-1)+1$. Additionally, since we chose $\lambda$ sufficiently small with respect to $k$, we see that
	\[
		\frac{1}{r-2k+3} + 130 k \sqrt \lambda \leq \frac{1}{r(k,k-1)+1} + 130 k \sqrt \lambda < \frac{1}{r(k,k-1)}.
	\]
	Therefore, the yellow edge density in $T$ is less than $\frac{1}{r(k,k-1)}$.
	We now argue identically to the proof of \cref{claim:no-r-or-b-sets}: by \cref{lem:3-color-supersaturation}, $T$ contains many blue $K_k$ or many red $K_{k-1}$, both of which are impossible since $v$ and every vertex in $T$ have red and blue degrees at least $d$.
\end{proof}

For the rest of the proof, let $G$ be the graph of red and blue edges in the coloring, and note that $G$ has minimum degree at least $(1- \frac{1}{r-1} -65k\sqrt \lambda)n$. In order to apply \cref{lem:stability} to $G$, we need to check that $G$ has few copies of $K_r$, which is done analogously to \cref{claim:few-Kk}.
\begin{claim}\label{claim:few-kr}
	$G$ has at most $2\tau (n)_r$ copies of $K_r$.
\end{claim}
\begin{proof}
	Every copy of $K_r$ in $G$ yields at least one red or blue copy of $K_k$ in the original coloring. By the same averaging argument as in the proof of \cref{lem:3-color-supersaturation}, if there are at least $2\tau(n)_r$ copies of $K_r$ in $G$, then there are at least $2 \tau(n)_k$ red or blue copies of $K_k$ in the original coloring. Without loss of generality, at least half of these are blue, so there are at least $\tau(n)_k$ labeled blue $K_k$. However, since every vertex of $K_n$ lies in $V_{RB} \cup V_{RBY}$, and thus has blue degree at least $d$, we obtain a contradiction to \cref{lem:3-color-few-H0}.
\end{proof}

We can now apply \cref{lem:stability} to $G$, and find that $G$ is nearly $(r-1)$-partite: it has a partition into parts $V_1,\dots,V_{r-1}$ such that the total number of internal edges in these parts is at most $\varepsilon \binom n2$. As in \cref{sec:proof-main}, we can assume that $V_1,\dots,V_{r-1}$ is a max $(r-1)$-cut of $G$. The following additional properties of the partition are proved identically to \cref{lem:partition-properties}: each part $V_i$ has size $\frac{n}{r-1} \pm \sqrt{2 \varepsilon}n$, and $e_{RB}(V_i,V_j) \geq (1-r^2 \varepsilon)\ab{V_i} \ab{V_j}$ for all $1 \leq i<j \leq r-1$, where $e_{RB}$ denotes the number of edges in $G$, that is, the number of red or blue edges.

The following claim is the three-color analogue of \cref{lem:high-red-deg}. The proof proceeds along the same lines as that of \cref{lem:high-red-deg}, but in this three-color setting, we need to split into two cases.
\begin{claim}\label{claim:3-color-high-red-deg}
	Fix a vertex $v \in V(K_n)$, and let $U_i = N_{RB}(v) \cap V_i$ denote the set of red or blue neighbors of $v$ inside $V_i$, for $1 \leq i \leq r-1$. Then there exists some $i$ such that $\ab{U_i} \leq \theta \ab{V_i}$.
\end{claim}
\begin{proof}
	Suppose for contradiction that $\ab{U_i} \geq \theta \ab{V_i}$ for all $1 \leq i \leq r-1$. Then for all $i \neq j$, we have that
	\begin{equation*}
		e_{RB} (U_i,U_j) \geq \ab{U_i} \ab{U_j} - r^2 \varepsilon \ab{V_i} \ab{V_j} \geq \left( 1- \frac{r^2 \varepsilon}{\theta^2} \right) \ab{U_i} \ab{U_j}.
	\end{equation*}
	This implies that if we pick uniformly random vertices from $U_i$ and $U_j$, they are joined by a red or blue edge with probability at least $1-r^2 \varepsilon/\theta^2$. By the union bound, this implies that if we pick a uniformly random vertex from each $U_i$, the probability that they form a $K_{r-1}$ with no yellow edges is at least $1- \binom{r-1}2 r^2 \varepsilon/\theta^2 \geq \frac 12$, since we chose $\varepsilon \ll \theta \ll \frac 1r$. Therefore, the number of copies of $K_{r-1}$ in $S \coloneqq U_1 \cup \dotsb \cup U_{r-1}$ with no yellow edges is at least
	\begin{align*}
		\frac 12 \prod_{i=1}^{r-1} \ab{U_i} &\geq \frac{\theta^{r-1}}{2} \prod_{i=1}^{r-1} \ab{V_i} \geq  \frac{\theta^{r-1}}{2} \left( \frac{n}{r-1} - \sqrt{2\varepsilon}n \right)^{r-1} \geq \beta (n)_{r-1},
	\end{align*}
	where $\beta>0$ is a constant depending only on $k$ (as it depends only on $\theta,r,$ and $\varepsilon$, each of which depends only on $k$). Recall that $\lambda$ also depends only on $k$, and that $\tau=(1+\lambda)^{-t}$, which decays exponentially in $t$ for fixed $\lambda$. Therefore, since we picked $t$ sufficiently large, we have that $\beta \geq (2t+2)\tau$, as the right-hand side also decays exponentially in $t$.
	In all, we conclude that the number of copies of $K_{r-1}$ in $S$ with no yellow edges is at least $(2t+2)\tau (n)_{r-1}$.
	Since $S \subseteq N_{RB}(v)$, every such $K_{r-1}$ yields a copy of $K_r$ containing $v$ with no yellow edges. Let $\Q$ denote this set of copies of $K_{r}$.

	For every clique $Q \in \Q$, we fix some red or blue $K_k$ contained in $Q$, and call it $Q'$. Let $\Q_1 \subseteq \Q$ be the set of $Q \in \Q$ such that $v \notin Q'$. Note that every red or blue $K_k$ can appear as $Q'$ for at most $(n-k-1)_{r-k-1}$ choices of $Q \in \Q_1$, for having fixed the red or blue $K_k$, we need to pick $r-k-1$ other vertices besides $v$ to complete to a $K_r$ in $\Q_1$. Therefore, if $\ab{\Q_1} \geq 2 \tau (n)_{r-1}$, then we find at least $2\tau(n)_{k}$ red or blue $K_k$ not containing $v$. At least half of these are blue (say), which contradicts \cref{lem:3-color-few-H0}, since each of these blue $K_k$ is contained in $V_{RB} \cup V_{RBY}$, and thus all of their vertices have at least $d$ blue neighbors.

	Therefore, if we let $\Q_2 = \Q \setminus \Q_1$, we conclude that $\ab{\Q_2} \geq 2t\tau (n)_{r-1}$. By the same argument as in the last paragraph, this implies that there are at least $2t\tau(n)_{k-1}$ red or blue $K_k$ containing $v$. Without loss of generality at least half of them are blue. Since every vertex has blue degree at least $d$, every such blue $K_k$ extends to at least $(d-k)_{t-k}$ labeled blue copies of $H$, which contradicts the bound \eqref{eq:local-ramsey-blowup-ub}.
\end{proof}

The following claim is the three-color analogue of \cref{lem:vertices-ok}.
\begin{claim}\label{claim:3-col-vertices-ok}
	Let $v \in V(K_n)$, and suppose that $v$ lies in part $V_i$. Then $\ab{N_Y(v) \cap V_i}\geq (1- \theta)\ab{V_i}$ and $\ab{N_{RB}(v) \cap V_j} \geq (1-2r \theta)\ab{V_j}$ for all $j \neq i$.
\end{claim}

\cref{claim:3-col-vertices-ok} is proved by combining the proofs of \cref{lem:high-blue-deg,lem:vertices-ok}, with no new ideas; in particular, one does not need to split into cases as in the proof of \cref{claim:3-color-high-red-deg}.

We now know that almost all edges inside each part $V_i$ are yellow, and that almost all edges between parts $V_i,V_j$ are red or blue. To conclude the proof, it remains to eliminate these ``almost''s. The next claim shows that each part is monochromatic yellow, establishing the three-color analogue of \cref{lem:parts-red}. The heart of the proof is the same as that of \cref{lem:parts-red}, but there are a few more cases to consider in this three-color setting.

\begin{claim}
	For every $1 \leq i \leq r-1$, all edges inside $V_i$ are yellow.
\end{claim}
\begin{proof}
	By relabeling the parts, we may assume that $i=1$. So suppose for contradiction that there exist $u,v \in V_1$ such that the edge $uv$ is red (the case where it is blue follows identically). For each $j>1$, let $W_j=N_{RB}(u) \cap N_{RB}(v) \cap V_j$ denote the set of vertices in $V_j$ which are common neighbors of $u$ and $v$ in $G$. By \cref{claim:3-col-vertices-ok}, we know that $\ab{W_j} \geq (1-4r \theta) \ab{V_j}\geq \frac 12 \ab{V_j}$ for all $j>1$, which implies that for all $2 \leq i \neq j \leq k-1$, we have
	\[
		e_{RB}(W_i,W_j) \geq \ab{W_i} \ab{W_j} - r^2 \varepsilon \ab{V_i}\ab{V_j} \geq (1-4r^2 \varepsilon) \ab{W_i} \ab{W_j}.
	\]
	Therefore, if we pick a random vertex $w_j$ from each $W_j$, then by the union bound, they will form a copy of $K_{r-2}$ in $G$ with probability at least $1- \binom{r-2}2 \cdot 4r^2 \varepsilon \geq \frac 12$. Thus, we find that the number of
	copies of $K_r$ in $G$ containing the vertices $u$ and $v$ is at least
	\[
		\frac 12 \prod_{j=2}^{r-1} \ab{W_j} \geq \frac 12 (1-4r \theta)^{r-2} \prod_{j=2}^{r-1} \ab{V_j} \geq \frac 12 (1-4r \theta)^{r-2} \left(\frac{n}{r}\right)^{r-2} \geq 4 r^{-r} n^{r-2},
	\]
	since we picked $\theta$ sufficiently small with respect to $k$, and thus with respect to $r$, and since $r \geq 4$.
	Let $\Q$ denote this collection of $K_r$.

	For every clique $Q\in \Q$, we fix some red or blue $K_k$ contained in $Q$, and call it $Q'$. Let $\Q_1 \subseteq \Q$ be the set of $Q \in \Q$ such that $Q'$ does not contain either $u$ or $v$. Note that every red or blue $K_k$ can appear as $Q'$ for at most $n^{r-k-2}$ choices of $Q \in \Q_1$, for having fixed the vertices of the $K_k$, we need to pick $r-k-2$ other vertices besides $u$ and $v$ to complete to a $K_r$ in $\Q_1$. Therefore, if $\ab{\Q_1} \geq r^{-r} n^{r-2}$, then we find at least $r^{-r} n^{k}$ red or blue $K_k$ not containing $u$ or $v$. At least half of these are blue (say), which contradicts \cref{lem:3-color-few-H0}, since we have found many blue $K_k$ in $V_{RB}\cup V_{RBY}$.

	Next, let $\Q_2$ denote the set of $Q \in \Q$ such that $Q'$ contains $u$ but not $v$. By the same reasoning, if $\ab{\Q_2} \geq r^{-r} n^{r-2}$, then we can find at least $r^{-r} n^{k-1}$ monochromatic $K_k$ containing $u$ but not $v$. Say that at least half of these are blue. Each of them extends to at least $(d-t)_{t-k-1}$ blue labeled copies of $H$ containing $u$, since every vertex has blue degree at least $d$. As in the proof of \cref{claim:no-r-or-b-sets,claim:no-vtx-high-yellow-deg}, this is a contradiction to  the bound (\ref{eq:local-ramsey-blowup-ub}).
	By the same reasoning, we see that $\ab{\Q_3} < r^{-r} n^{r-2}$, where $\Q_3$ denotes the set of $Q \in \Q$ such that $Q'$ contains $v$ but not $u$.

	Therefore, if we let $\Q_4$ denote the set of $Q \in \Q$ such that $Q'$ contains both $u$ and $v$, we find that $\ab{\Q_4} \geq r^{-r} n^{r-2}$. By the same averaging as above, this implies that there are at least $r^{-r} n^{k-2}$ red $K_k$ which contain the edge $uv$ (we know they must be red because the edge $uv$ is red). Moreover, by \cref{claim:no-very-low-r-b-deg}, every vertex in such a red $K_k$ has red degree at least $\frac{2k-6}{r-1}n \geq \frac{3}{r-1}n$, so such a red $K_k$ extends to at least $(\frac{3}{r-1}n - k)_{t-k} =(1-o(1))(\frac{3}{r-1})^{t-k} (n)_{t-k}$ labeled red copies of $H$. So in total, the edge $uv$ lies in at least
	\[
		(1-o(1)) r^{-r} \left(\frac{3}{r-1}\right)^{t-k} (n)_{t-2}
	\]
	labeled red copies of $H$. Now we create a new coloring $\chi'$ by recoloring the edge $uv$ yellow, and estimate how many yellow copies of $H$ are produced. There are at most $k!(n)_{k-2}$ labeled yellow $K_k$ containing $u$ and $v$, and each of these extends to a labeled yellow copy of $H$ in at most $(\frac{2}{r-1}n)_{t-k}$ ways, since every vertex has yellow degree at most $(\frac{1}{r-1}+65k\sqrt \lambda)n \leq \frac{2}{r-1}n$ by \cref{claim:no-vtx-high-yellow-deg}. Similarly, there are at most $2k!(n)_{k-1}$ labeled yellow $K_k$ containing exactly one of $u$ or $v$, and each of these extends to a labeled yellow copy of $H$ containing both $u$ and $v$ in at most $(\frac{2}{r-1}n)_{t-k-1}$ ways. In total, by recoloring $uv$ yellow, we create at most
	\[
		k! (n)_{k-2} \left(\frac{2}{r-1}n\right)_{t-k} + 2k! (n)_{k-1} \left(\frac{2}{r-1}n\right)_{t-k-1} \leq (1+o(1)) 3k! \left(\frac{2}{r-1}\right)^{t-k-1} (n)_{t-2}
	\]
	labeled yellow copies of $H$. Since we chose $t$ sufficiently large with respect to $k$, this is less than the number of red copies of $H$ we destroy by recoloring $uv$ yellow, which shows that $\chi'$ has strictly fewer monochromatic copies of $H$ than $\chi$, a contradiction.
\end{proof}

We now know that every part $V_i$ is a monochromatic yellow clique. The final claim, analogous to \cref{lem:across-blue}, says that all the remaining edges are red or blue, and that there are no red or blue $K_k$.
\begin{claim}\label{claim:no-yellow-across}
	There is no yellow edge between $V_i$ and $V_j$ for any $1 \leq i \neq j \leq r-1$. Additionally, there is no red or blue copy of $K_k$.
\end{claim}
\begin{proof}
	Suppose that there is some yellow edge between $V_i$ and $V_j$. Since every edge in $V_i$ is yellow and $\ab{V_i}>t$, this edge participates in at least one yellow copy of $H$. Consider the coloring $\chi'$ defined by making all edges within the parts yellow, and making the edges between the parts be a blowup of a Ramsey coloring of $K_{r-1}$. This coloring has no red or blue $H$ (since there is not even a red or blue $K_k$), and the number of yellow $H$ is strictly smaller than in $\chi$, since there are exactly as many yellow $H$ within the parts, but none containing vertices from two different parts. This contradicts the minimality of $\chi$.

	Now suppose that there is some red or blue $K_k$ in $\chi$. In this case, we again see that $\chi'$ has strictly fewer monochromatic copies of $H$ than $\chi$: it has the same number of yellow copies of $H$, but strictly fewer red or blue copies because we destroyed at least one red or blue $K_k$.
\end{proof}
Therefore, we have found that each $V_i$ is a monochromatic yellow clique, that all remaining edges are red or blue, and that there is no red or blue $K_k$. The final step of the proof is the same as that of \cref{thm:main}: Jensen's inequality shows that any coloring of this type which minimizes the number of yellow copies of $H$ is one in which the parts are as equally-sized as possible, completing the proof.
\end{proof}

\begin{rem}
	For two graphs $F_1,F_2$, let $r(F_1,F_2)$ be the Ramsey number of $F_1,F_2$, namely the least $r$ so that any red/blue coloring of $E(K_r)$ contains a red copy of $F_1$ or a blue copy of $F_2$.
	Examining the proof of \cref{thm:3-color}, one can see that it holds in greater generality than simply appending pendant edges to cliques, but it still holds in a much more limited setting than \cref{thm:main}. Namely, let us say that a graph $H_0$ is \emph{polite} if there exists a partition of the vertex set of $H_0$ into two induced subgraphs $H_1,H_2$ so that for every $i \in \{1,2\}$ and every $v \in V(H_0)$, we have the bounds
	\[
		\frac{r(H_0,H_i)}{r(H_0,H_0)} \leq 2^{-23}
	\]
	and
	\[
		\frac{r(H_0,H_0)-1}{r(H_0, H_0 \setminus \{v\})} \geq 1+ 25 \left(\frac{r(H,H_i)}{r(H,H)}\right)^{1/4}.
	\]
	Thus, an integer $k$ is polite in the sense of \cref{def:polite} if and only if $K_k$ is polite. 

	Now, suppose that $k$ is sufficiently large, and let $H_0$ be a polite $k$-critical graph containing a copy of $K_k$. Then one can check that the proof of \cref{thm:3-color} carries through, namely that if one appends sufficiently many pendant edges to $H_0$, then the resulting graph $H$ is a three-color bonbon. However, we chose not to state and prove the theorem in this generality, simply because we don't expect it to apply to a particularly rich class of graphs. 
\end{rem}

\section{Concluding remarks}

\subsection{On the definition of three-color bonbons}\label{sec:weaker-bonbon-def}
Although the definition of a three-color bonbon is a natural generalization of that of a two-color bonbon, it is perhaps not the natural generalization one would first guess. As a consequence, there is also a mismatch between the statements of \cref{thm:main,thm:3-color}. In the case of two colors, we precisely characterized the extremal coloring minimizing the number of monochromatic copies of $H$: the Tur\'an coloring is the only such coloring. In the case of three colors, we only (conditionally) proved a result that appears weaker: that any extremal coloring has an equitable vertex partition into $r(k)-1$ yellow cliques and has no red or blue $K_k$, but we did not prove that the Ramsey-blowup coloring is the unique extremizer. Here, and throughout this section, we use the notation $r(k)\coloneqq r(k,k)$ to denote the diagonal Ramsey number.

In fact, we do not believe that there is a unique extremizer in the three-color case. More generally, we believe that there should exist red/blue colorings of a blowup of $K_{r(k)-1}$ which do not contain a red or a blue $K_k$, but are not a blowup of a Ramsey coloring, that is, a two-coloring of $E(K_{r(k)-1})$ without a monochromatic $K_k$. Given a graph $G$ and a positive integer $s$, we denote by $G[s]$ the $s$-blowup of $G$.
\begin{Def}
	Let $k$ and $s$ be positive integers. We call a red/blue coloring of $E(K_{r(k)-1}[s])$ a \emph{mixed blowup coloring} if it contains no monochromatic $K_k$, but is not the blowup of a Ramsey coloring of $E(K_{r(k)-1})$.
\end{Def}

If there exist mixed blowup colorings of $K_{r(k)-1}[s]$, then the stronger alternative definition of a three-color bonbon---that the Ramsey-blowup coloring is the unique extremal coloring---is too strong. Indeed, any mixed blowup coloring can be used as the red and blue edges of a three-coloring of $K_n$, and it will contain the same number of monochromatic $H$ as the Ramsey-blowup coloring, where $H$ is obtained from $K_k$ by appending pendant edges.

As it turns out, there is a simple characterization of when mixed blowup colorings exist.
\begin{lem}\label{lem:mixed-equivalence}
	Let $k\geq 3$ be an integer and $r\coloneqq r(k)$. The following statements are equivalent.
	\begin{enumerate}[label=(\roman*)]
		\item There exists a two-coloring of $E(K_{r-2})$ which extends to two distinct two-colorings of $E(K_{r-1})$, neither of which contains a monochromatic $K_k$.\label{item:two-extensions}
		\item For every $s \geq 2$, there exists a mixed blowup coloring of $K_{r-1}[s]$.\label{item:mixed-exist}
		\item There exists a mixed blowup coloring of $K_{r-1}[2]$.\label{item:mixed-2}
	\end{enumerate}
\end{lem}
\begin{proof}
	First, suppose that \ref{item:two-extensions} holds, and let $\chi_1,\chi_2$ be two distinct Ramsey colorings of $E(K_{r-1})$, and let $u\in V(K_{r-1})$ be a vertex such that $\chi_1$ and $\chi_2$ agree on $K_{r-1} \setminus \{u\}$. For any $s \geq 2$, consider the blowup $K_{r-1}[s]$, and let $U$ be the part of $V(K_{r-1}[s])$ corresponding to the vertex $u$. Between all pairs of vertices of $K_{r-1}[s]$ which are not in $U$, we color by blowing up $\chi_1$ (or $\chi_2$, since they agree away from $u$). We then arbitrarily partition $U$ into two non-empty sets $U_1,U_2$, and color all edges incident to $U_1$ by blowing up $\chi_1$, and all edges incident to $U_2$ by blowing up $\chi_2$. The resulting coloring is not a Ramsey-blowup coloring, since $\chi_1$ and $\chi_2$ are distinct. Additionally, it has no monochromatic $K_k$, since neither $\chi_1$ nor $\chi_2$ has a monochromatic $K_k$, so it is a mixed blowup coloring. This shows that \ref{item:two-extensions} implies \ref{item:mixed-exist}.

	It is immediate that \ref{item:mixed-exist} implies \ref{item:mixed-2}, so assume that \ref{item:mixed-2} holds. Fix a mixed blowup coloring of $K_{r-1}[2]$. Since this is not a blowup coloring, there must exist two vertices $u,u'$ in one part $U$ whose incident edges are not colored identically. Let $v_1$ be a vertex in some part $U_1$ such that the edges $uv_1$ and $u'v_1$ receive different colors, and let $v_2,\dots,v_{r-2}$ be arbitrary vertices from the $r-3$ parts other than $U$ and $U_1$. Then $v_1,v_2,\dots,v_{r-2}$ 
	span a coloring of $K_{r-2}$ which extends to two distinct Ramsey colorings of $K_{r-1}$, by choosing either $u$ or $u'$ as the extension. This shows that \ref{item:two-extensions} holds, and completes the proof.
\end{proof}

We conjecture that mixed blowup colorings exist for infinitely many $k$.
\begin{conj}\label{conj:mixed}
	There exist a mixed blowup coloring of $K_{r(k)-1}[2]$ for infinitely many $k$.
\end{conj}

Thanks to \cref{lem:mixed-equivalence}, to find a mixed blowup coloring, it suffices to exhibit two Ramsey colorings of $K_{r(k)-1}$ which differ on a single vertex. Unfortunately, we are not able to do this for any value of $k$. A big part of the problem is doing this requires knowing the value of $r(k)$, and this is only known for $k=3,4$. Additionally, for both of these values, it is known that there is a unique coloring of $E(K_{r(k)-1})$ with no monochromatic $K_k$, and in this unique Ramsey coloring, all vertices have the same red and blue degree. This implies that any coloring of $K_{r(k)-2}$ vertices extends to a Ramsey coloring of $K_{r(k)-1}$ in at most one way.
For a similar reason, we only conjecture that mixed colorings exist for infinitely many $k$, as opposed to existing for all sufficiently large $k$. Indeed, it seems plausible that for infinitely many very special values of $k$, there is a unique, highly structured Ramsey coloring (e.g.\ one coming from a Paley graph, as happens for $k=3$ and $k=4$). For such $k$, there should not exist mixed blowup colorings. However, for ``most'' $k$, we expect there to be several different Ramsey colorings, among which one can likely find two that differ on a single vertex.

For off-diagonal Ramsey numbers, we can prove the existence of mixed blowup colorings. Indeed, it is known (see e.g.\ \cite{MR1670625}) that $r(3,4)=9$, and that there are exactly three non-isomorphic colorings of $K_8$ with no red $K_3$ and no blue $K_4$. The following figure shows these three colorings, where the edges correspond to red edges, the non-edges correspond to blue edges, and the dashed edges can be colored either red or blue. In particular, we find that there are distinct colorings that differ only on one vertex, namely any endpoint of one of the dashed edges.
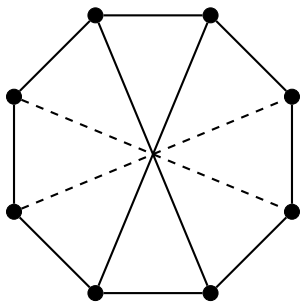
\begin{center}
	\begin{tikzpicture}
		\foreach \x in {1,...,8} \node[vert] (\x) at (45*\x+22.5:2) {};
		\foreach \x [remember=\x as \lastx (initially 8)] in {1,...,8} \draw[thick] (\lastx) -- (\x);
		\draw[thick] (1)--(5);
		\draw[thick] (2)--(6);
		\draw[thick,dashed] (3)--(7);
		\draw[thick,dashed] (4)--(8);
	\end{tikzpicture}
\end{center}

It is known that there are at least 328 colorings of $K_{42}$ without a monochromatic $K_5$, and in \cite[Conjecture 2]{MR1438619}, it is conjectured that $r(5,5)=43$ and that these 328 are the only Ramsey colorings for $K_5$. If this conjecture is true, then it establishes the existence of mixed blowup colorings for $k=5$. Indeed, several of these colorings of $K_{42}$ contain edges that can be colored in one of two ways. For example, using McKay's list \cite[\texttt{r55\_42some.g6}]{McKayList} of Ramsey (5,5,42) graphs, one can readily check that deleting the edge $(31, 39)$ from the first graph yields a graph isomorphic to the 13th graph. In other words, these two colorings differ only on a single edge, and in particular on a single vertex.

For completeness, here are these two $42$-vertex graphs, which differ in a single edge and which both have clique and independence numbers $4$, presented in the \texttt{graph6} format.
\begin{itemize}
	\item 
	\begin{verbatim}
		i?Udjp^j}?W@`bIRhHgk\SY~ECeQS\CniuKP]RQLdsX~F?b|L?h_SvygSNziSVdZ`P|
		CxamFHKax[PhPyVEYxAqkY\_xCfYxNscNtb]k_uFsLruaJwr`nPMMc]\qGhwyh
		fLjTELQ}T]h@qtuW
	\end{verbatim}
	\item 
	\begin{verbatim}
		i?Udjp^j}?W@`bIRhHgk\SY~ECeQS\CniuKP]RQLdsX~F?b|L?h_SvygSNziSVdZ`P|
		CxamFHKax[PhPyVEYxAqkY\_xCfYxNscNtb]k_uFsLruaJwr`nPMMc]\qGhwyh
		dLjTELQ}T]h@qtuW
	\end{verbatim}
\end{itemize}
The only difference between the strings is the first character of the third row in both cases, corresponding to the fact that the two graphs differ in only one edge.

\subsection{Even more colors}\label{sec:more-colors}
Let $k$ be a large integer, and let $H$ be obtained from $K_k$ by adding sufficiently many pendant edges.
Given that we proved that the Ramsey multiplicity behavior of $H$ is determined by a blowup coloring of a Ramsey coloring on one fewer color for both two and three colors, it is natural to expect the same behavior to persist for arbitrarily many colors. However, there appear to be major obstacles to proving such a thing, and indeed, it may not be true. Firstly, our proof technique seems to fail right away when dealing with at least four colors: in the step where we partition the vertices according to which colors they have high degree in, it is not clear how to prove that all but one of these sets is small.

To see this, we recall the simple fact, usually attributed to Lefmann \cite{MR932230}, that for any integers $k,q_1,q_2$,
\begin{equation}
	r_{q_1 + q_2}(k) -1 \geq (r_{q_1}(k)-1)(r_{q_2}(k)-1).\label{eq:product-coloring}
\end{equation}
Indeed, given an optimal $q_1$-edge coloring of $K_{r_{q_1}(k)-1}$ and an optimal $q_2$-edge coloring of $K_{r_{q_2}(k)-1}$, we can form a $(q_1+q_2)$-coloring of $K_{(r_{q_1}(k)-1)(r_{q_2}(k)-1)}$ by taking a lexicographic product, which will have no monochromatic $K_k$. Equivalently, we can blow up the $q_1$-coloring to parts of size $r_{q_2}(k)-1$, and then color each part according to the $q_2$-coloring.

Let us suppose that there exist some $k,q_1,q_2$ for which inequality (\ref{eq:product-coloring}) is actually an equality. In that case, there are many non-isomorphic colorings of $K_n$ with $q_1+q_2+1$ colors, all of which yield the same bound on the Ramsey multiplicity constant of $H$. Indeed, we may first equitably partition $K_n$ into $r_{q_1}(k)-1$ parts, and color the edges between these parts according to a blowup of the $q_1$-coloring. Inside each part, we pick any $q_2$-subset of the remaining $q_2+1$ colors, and color according to a blowup of the $q_2$-coloring. Finally, inside each sub-part, we use the remaining color to form a monochromatic clique of size $n/((r_{q_1}(k)-1)(r_{q_2}(k)-1))$. If we make the same choice inside each top-level part, we get the Ramsey-blowup coloring. However, if we make different choices inside each part, we'll obtain another coloring yielding the same multiplicity bound, which is Ramsey-blowup-like. Indeed, in any Ramsey-blowup-like coloring, all copies of $H$ have the same color, whereas here they may have different colors. 

All of this works under the assumption that (\ref{eq:product-coloring}) is tight, which may seem like a very strong assumption. Nonetheless, until very recently \cite{MR4246789,2105.08850}, the best known lower-bound constructions for $r_{q_1+q_2}(k)$ for $q_1+q_2 \geq 5$ were of this product form. Moreover, it is a major open problem (see e.g.\ \cite{Alon,MR1815606}) to determine whether $r_q(k)$ grows exponentially or super-exponentially as a function of $q$ (for fixed $k$); the question of whether (\ref{eq:product-coloring}) is tight is a special case of this question. 

Thus, it seems as though proving that $H$ is a $q$-color bonbon for $q \geq 5$ is likely to be very difficult, and it is possible that there do not exist any $q$-color bonbons for $q \geq 5$.

\subsection{Other open problems}\label{sec:open-problems}

Recall that in \cref{thm:main}, we prove that we obtain a bonbon if we add $t \geq (1000kh)^{10} h^{10k}$ pendant edges to a $k$-critical graph $H_0$ with $h$ vertices. We made no effort to optimize the constants in the lower bound on $t$, though our proof technique does require $t$ to be at least exponentially large in $k$. Moreover, as discussed in the introduction, if $H_0=K_k$, then some lower bound on $t$ is necessary: if we add $o(k^2/{\log k})$ pendant edges to $K_k$, the result will not be a bonbon, since the random coloring yields a stronger upper bound on $c(H)$ than the Tur\'an coloring. 
It would be interesting to determine what the correct lower bound on $t$ is; for instance, in case $H_0=K_k$, is $H$ a bonbon even if $t$ is only polynomial in $k$?

It would also be very interesting to prove that other graphs are bonbons (in any number of colors). A natural place to start is in the family of graphs that we call \emph{generalized lollipops}. We say that a graph $H$ is a $(k,t)$-generalized lollipop if it has $t$ vertices and contains a $K_k$ whose deletion yields a forest. Equivalently, a $(k,t)$-generalized lollipop is obtained from $K_k$ by attaching trees comprising $t-k$ total vertices to the vertices of the $K_k$. For all such graphs, the Tur\'an coloring yields an upper bound on their Ramsey multiplicity constant of $(k-1)^{1-t}$.
\begin{conj}\label{conj:generalized-lollipop}
	If $k \geq 4$ and $t$ is sufficiently large in terms of $k$, then any $(k,t)$-generalized lollipop is a two-color bonbon.
\end{conj}
A natural special case of this conjecture is interesting in its own right, and may be easier to prove than the full \cref{conj:generalized-lollipop}. Namely, the \emph{lollipop} graph $L_{k,t}$ which is obtained from $K_k$ by appending to a single vertex a path with $t-k$ edges.
\begin{conj}\label{conj:path-lollipop}
	If $k \geq 4$ and $t$ is sufficiently large in terms of $k$, then $L_{k,t}$ is a two-color bonbon.
\end{conj}
Of course, one could also pose versions of \cref{conj:generalized-lollipop,conj:path-lollipop} for three-color bonbons, but we expect such conjectures to be even harder to resolve.

\section*{Acknowledgments} 
We are indebted to Jon Noel for pointing out an error in  \cref{lem:degree-regular,lem:3-color-degree-regular} in an earlier version of this paper. We would also like to thank Hao Huang for interesting discussions which led to the inclusion of \cref{subsec:edge-deletion}. Finally, we are grateful to the anonymous referees for helpful comments that greatly improved the presentation of this paper.





\begin{aicauthors}
\begin{authorinfo}[jacob]
  Jacob Fox\\
  Department of Mathematics, Stanford University\\ 
  Stanford, CA, USA\\
  jacobfox\imageat{}stanford\imagedot{}edu \\
  \url{https://stanford.edu/~jacobfox/}
\end{authorinfo}
\begin{authorinfo}[yuval]
  Yuval Wigderson\\
  School of Mathematics, Tel Aviv University\\
  Tel Aviv 69978, Israel\\
  yuvalwig\imageat{}tauex\imagedot{}tau\imagedot{}ac\imagedot{}il \\
  \url{http://www.math.tau.ac.il/~yuvalwig/}
\end{authorinfo}
\end{aicauthors}


\begin{thebibliography}{10}
  \providecommand{\url}[1]{\texttt{#1}}
  \providecommand{\urlprefix}{URL }
  \providecommand{\eprint}[2][]{\url{#2}}
  
  \bibitem{Alon}
  N.~Alon, Lov\'asz, vectors, graphs and codes, in I.~B\'ar\'any, G.~Katona, and
    A.~Sali (eds.), \emph{Building Bridges {II}}, \emph{Bolyai Soc. Math. Stud.},
    vol.~28, Springer, 2019,  1--16.
  
  \bibitem{MR2370517}
  B.~Bollob\'{a}s and V.~Nikiforov, Joints in graphs, \emph{Discrete Math.}
    \textbf{308} (2008), 9--19.
  
  \bibitem{MR992396}
  S.~A. Burr, P.~Erd\H{o}s, R.~J. Faudree, and R.~H. Schelp, On the difference
    between consecutive {R}amsey numbers, \emph{Utilitas Math.} \textbf{35}
    (1989), 115--118.
  
  \bibitem{MR595601}
  S.~A. Burr and V.~Rosta, On the {R}amsey multiplicities of graphs---problems
    and recent results, \emph{J. Graph Theory} \textbf{4} (1980), 347--361.
  
  \bibitem{MR609100}
  V.~Chv\'{a}tal and E.~Szemer\'{e}di, On the {{E}rd\H{o}s}-{S}tone theorem,
    \emph{J. London Math. Soc. (2)} \textbf{23} (1981), 207--214.
  
  \bibitem{MR2927637}
  D.~Conlon, On the {R}amsey multiplicity of complete graphs,
    \emph{Combinatorica} \textbf{32} (2012), 171--186.
  
  \bibitem{MR4115773}
  D.~Conlon, The {R}amsey number of books, \emph{Adv. Comb.}  (2019), Paper No.
    3, 12.
  
  \bibitem{MR2738996}
  D.~Conlon, J.~Fox, and B.~Sudakov, An approximate version of {S}idorenko's
    conjecture, \emph{Geom. Funct. Anal.} \textbf{20} (2010), 1354--1366.
  
  \bibitem{CoFoWi}
  D.~Conlon, J.~Fox, and Y.~Wigderson, Ramsey numbers of books and
    quasirandomness, \emph{Combinatorica} \textbf{42} (2022), 309--363.
  
  \bibitem{2111.05420}
  D.~Conlon, J.~Fox, and Y.~Wigderson, Three early problems on size {R}amsey
    numbers, \emph{Combinatorica}  (2023), to appear. Preprint available at
    arXiv:2111.05420.
  
  \bibitem{MR3893193}
  D.~Conlon, J.~H. Kim, C.~Lee, and J.~Lee, Some advances on {S}idorenko's
    conjecture, \emph{J. Lond. Math. Soc. (2)} \textbf{98} (2018), 593--608.
  
  \bibitem{MR4237083}
  D.~Conlon and J.~Lee, Sidorenko's conjecture for blow-ups, \emph{Discrete
    Anal.}  (2021), Paper No. 2, 13.
  
  \bibitem{MR3071377}
  J.~Cummings, D.~Kr\'{a}\v{l}, F.~Pfender, K.~Sperfeld, A.~Treglown, and
    M.~Young, Monochromatic triangles in three-coloured graphs, \emph{J. Combin.
    Theory Ser. B} \textbf{103} (2013), 489--503.
  
  \bibitem{MR151956}
  P.~Erd\H{o}s, On the number of complete subgraphs contained in certain graphs,
    \emph{Magyar Tud. Akad. Mat. Kutat\'{o} Int. K\"{o}zl.} \textbf{7} (1962),
    459--464.
  
  \bibitem{MR183654}
  P.~Erd\H{o}s, On extremal problems of graphs and generalized graphs,
    \emph{Israel J. Math.} \textbf{2} (1964), 183--190.
  
  \bibitem{MR655605}
  P.~Erd\H{o}s, Some new problems and results in graph theory and other branches
    of combinatorial mathematics, in \emph{Combinatorics and graph theory
    ({C}alcutta, 1980)}, \emph{Lecture Notes in Math.}, vol. 885, Springer,
    Berlin-New York, 1981,  9--17.
  
  \bibitem{MR726456}
  P.~Erd\H{o}s and M.~Simonovits, Supersaturated graphs and hypergraphs,
    \emph{Combinatorica} \textbf{3} (1983), 181--192.
  
  \bibitem{MR19911}
  P.~Erd\"{o}s, Some remarks on the theory of graphs, \emph{Bull. Amer. Math.
    Soc.} \textbf{53} (1947), 292--294.
  
  \bibitem{MR1556929}
  P.~Erd\"{o}s and G.~Szekeres, A combinatorial problem in geometry,
    \emph{Compositio Math.} \textbf{2} (1935), 463--470.
  
  \bibitem{MR2374234}
  J.~Fox, There exist graphs with super-exponential {R}amsey multiplicity
    constant, \emph{J. Graph Theory} \textbf{57} (2008), 89--98.
  
  \bibitem{2109.09205}
  J.~Fox, X.~He, and Y.~Wigderson, Ramsey goodness of books revisited, \emph{Adv.
    Comb.}  (2023), to appear. Preprint available at arXiv:2109.09205.
  
  \bibitem{2012.12646}
  D.~Gerbner, On {T}ur\'{a}n-good graphs, \emph{Discrete Math.} \textbf{344}
    (2021), Paper No. 112445, 8pp.
  
  \bibitem{2006.03756}
  D.~Gerbner and C.~Palmer, Some exact results for generalized {T}ur\'{a}n
    problems, \emph{European J. Combin.} \textbf{103} (2022), Paper No. 103519,
    13pp.
  
  \bibitem{MR107610}
  A.~W. Goodman, On sets of acquaintances and strangers at any party, \emph{Amer.
    Math. Monthly} \textbf{66} (1959), 778--783.
  
  \bibitem{GrLeLiVo}
  A.~Grzesik, J.~Lee, B.~Lidick\'y, and J.~Volec, On tripartite common graphs,
    \emph{Combin. Probab. Comput.}  (2022), 1--17.
  
  \bibitem{MR2607540}
  H.~Hatami, Graph norms and {S}idorenko's conjecture, \emph{Israel J. Math.}
    \textbf{175} (2010), 125--150.
  
  \bibitem{MR2959863}
  H.~Hatami, J.~Hladk\'{y}, D.~Kr\'{a}\v{l}, S.~Norine, and A.~Razborov,
    Non-three-colourable common graphs exist, \emph{Combin. Probab. Comput.}
    \textbf{21} (2012), 734--742.
  
  \bibitem{MR1394515}
  C.~Jagger, P.~\v{S}\v{t}ov\'{\i}\v{c}ek, and A.~Thomason, Multiplicities of
    subgraphs, \emph{Combinatorica} \textbf{16} (1996), 123--141.
  
  \bibitem{MR2866732}
  P.~Keevash, Hypergraph {T}ur\'{a}n problems, in \emph{Surveys in combinatorics
    2011}, \emph{London Math. Soc. Lecture Note Ser.}, vol. 392, Cambridge Univ.
    Press, Cambridge, 2011,  83--139.
  
  \bibitem{MR3456171}
  J.~H. Kim, C.~Lee, and J.~Lee, Two approaches to {S}idorenko's conjecture,
    \emph{Trans. Amer. Math. Soc.} \textbf{368} (2016), 5057--5074.
  
  \bibitem{2006.09422}
  D.~Kr\'{a}l', J.~A. Noel, S.~Norin, J.~Volec, and F.~Wei, Non-bipartite
    $k$-common graphs, \emph{Combinatorica} \textbf{42} (2022), 87--114.
  
  \bibitem{2206.05800}
  D.~Kr\'{a}\v{l}, J.~Volec, and F.~Wei, Common graphs with arbitrary chromatic
    number, 2022. Preprint available at arXiv:2206.05800.
  
  \bibitem{MR932230}
  H.~Lefmann, A note on {R}amsey numbers, \emph{Studia Sci. Math. Hungar.}
    \textbf{22} (1987), 445--446.
  
  \bibitem{1107.1153}
  J.~X. Li and B.~Szegedy, On the logarithmic calculus and {S}idorenko's
    conjecture, 2011. Preprint available at arXiv:1107.1153.
  
  \bibitem{McKayList}
  B.~D. McKay, Ramsey graphs. Available online at
    \url{http://users.cecs.anu.edu.au/~bdm/data/ramsey.html}.
  
  \bibitem{MR1438619}
  B.~D. McKay and S.~P. Radziszowski, Subgraph counting identities and {R}amsey
    numbers, \emph{J. Combin. Theory Ser. B} \textbf{69} (1997), 193--209.
  
  \bibitem{MR1815606}
  J.~Ne\v{s}et\v{r}il and M.~Rosenfeld, I. {S}chur, {C}. {E}. {S}hannon and
    {R}amsey numbers, a short story, \emph{Discrete Math.} \textbf{229} (2001),
    185--195.
  
  \bibitem{1207.4714}
  S.~Nie\ss, Counting monochromatic copies of {$K_4$}: a new lower bound for the
    {R}amsey multiplicity problem, 2012. Preprint available at arXiv:1207.4714.
  
  \bibitem{MR2520282}
  V.~Nikiforov and C.~C. Rousseau, Ramsey goodness and beyond,
    \emph{Combinatorica} \textbf{29} (2009), 227--262.
  
  \bibitem{2206.04036}
  O.~Parczyk, S.~Pokutta, C.~Spiegel, and T.~Szab\'o, New {R}amsey multiplicity
    bounds and search heuristics, 2022. Preprint available at arXiv:2206.04036.
  
  \bibitem{MR1670625}
  S.~P. Radziszowski, Small {R}amsey numbers, \emph{Electron. J. Combin.}
    \textbf{1} (originally published 1994, revised 2021), 116pp.
  
  \bibitem{2105.08850}
  W.~Sawin, An improved lower bound for multicolor {R}amsey numbers and a problem
    of {Erd\H{o}s}, \emph{J. Combin. Theory Ser. A} \textbf{188} (2022), Paper
    No. 105579, 11pp.
  
  \bibitem{MR1225933}
  A.~Sidorenko, A correlation inequality for bipartite graphs, \emph{Graphs
    Combin.} \textbf{9} (1993), 201--204.
  
  \bibitem{MR1033422}
  A.~F. Sidorenko, Cycles in graphs and functional inequalities, \emph{Mat.
    Zametki} \textbf{46} (1989), 72--79, 104. Translated to English in {\it Math.
    Notes} {\bf 46} (1989), 877--882.
  
  \bibitem{MR337690}
  M.~Simonovits, Extermal graph problems with symmetrical extremal graphs.
    {A}dditional chromatic conditions, \emph{Discrete Math.} \textbf{7} (1974),
    349--376.
  
  \bibitem{1106.1030}
  K.~Sperfeld, On the minimal monochromatic {$K_4$}-density, 2011. Preprint
    available at arXiv:1106.1030.
  
  \bibitem{1406.6738}
  B.~Szegedy, An information theoretic approach to {S}idorenko's conjecture,
    2014. Preprint available at arXiv:1406.6738.
  
  \bibitem{MR991659}
  A.~Thomason, A disproof of a conjecture of {Erd\H{o}s} in {R}amsey theory,
    \emph{J. London Math. Soc. (2)} \textbf{39} (1989), 246--255.
  
  \bibitem{MR4016583}
  H.~Topcu, S.~Sorgun, and W.~H. Haemers, The graphs cospectral with the
    pineapple graph, \emph{Discrete Appl. Math.} \textbf{269} (2019), 52--59.
  
  \bibitem{HypergraphSidorenko}
  Y.~Wigderson, Complete $r$-partite $r$-graphs are {S}idorenko: a brief
    exposition, 2021. Not intended for publication. Available at
    \url{http://www.math.tau.ac.il/~yuvalwig/math/expository/HypergraphSidorenko.pdf}.
  
  \bibitem{MR4246789}
  Y.~Wigderson, An improved lower bound on multicolor {R}amsey numbers,
    \emph{Proc. Amer. Math. Soc.} \textbf{149} (2021), 2371--2374.
  
  \bibitem{2208.11181}
  Y.~Wigderson, {R}amsey numbers upon vertex deletion, 2022. Preprint available
    at arXiv:2208.11181.
  
  \bibitem{MR2801235}
  X.~Xu, Z.~Shao, and S.~P. Radziszowski, More constructive lower bounds on
    classical {R}amsey numbers, \emph{SIAM J. Discrete Math.} \textbf{25} (2011),
    394--400.
  
  \end{thebibliography}
\end{document}